\documentclass[10pt,psamsfonts]{article}

\usepackage{amsmath, amsfonts, amsthm, amssymb, amscd, enumerate, url, slashed, stmaryrd, upgreek, thmtools, mathdots}
\usepackage[all,arc,color]{xy}
\usepackage{enumerate}
\usepackage{mathrsfs}
\usepackage{mathtools}
\usepackage[utf8]{inputenc}
\usepackage{tikz-cd}

   \usepackage[colorlinks = true,
            linkcolor = blue,
            urlcolor  = cyan,
            citecolor = red,
            anchorcolor = blue,
            pagebackref]{hyperref}
            
\usepackage[alphabetic,backrefs]{amsrefs}
 
 \usepackage[parfill]{parskip}
 \usepackage{anysize}
 \marginsize{1in}{1in}{1in}{1in}
 
\begingroup
 \makeatletter
 \@for\theoremstyle:=definition,remark,plain\do{%
 \expandafter\g@addto@macro\csname th@\theoremstyle\endcsname{%
 \addtolength\thm@preskip\parskip
 }%
 }
 \endgroup

\declaretheorem[name=Theorem,numberwithin=section]{thm}
\declaretheorem[name=Proposition,numberlike=thm]{prop}
\declaretheorem[name=Lemma,numberlike=thm]{lemma}
\declaretheorem[name=Corollary,numberlike=thm]{cor}

\declaretheorem[name=Definition,style=definition,qed=$\blacktriangle$,numberlike=thm]{defn}

\declaretheorem[name=Remark,style=definition,qed=$\blacktriangle$,numberlike=thm]{rmk}

\newcounter{noteCounter}
\DeclareMathOperator\vol{vol}

\newcommand{\imag}{\operatorname{Im}}
\newcommand{\G}{\mathrm{G}_2}

\newcommand{\GL}[1]{\mathrm{GL}(#1)}

\newcommand{\SO}[1]{\mathrm{SO}(#1)}

\newcommand{\so}{\mathfrak{so}}

\newcommand{\R}{\mathbb R}

\newcommand{\Qu}{\mathbb H}
\newcommand{\Oc}{\mathbb O}

\newcommand{\cB}{\mathcal{B}}

\newcommand{\ph}{\varphi}
\newcommand{\ps}{\psi}
\newcommand{\sta}{\star}
\newcommand{\hk}{\mathbin{\! \hbox{\vrule height0.3pt width5pt depth 0.2pt \vrule height5pt width0.4pt depth 0.2pt}}}

\newcommand{\lieg}{\mathfrak {g}_2}

\newcommand{\im}{\operatorname{im}}
\newcommand{\rank}{\operatorname{rank}}

\newcommand{\rest}[2]{ {\left. {#1} \right|}_{{#2} }}

\newcommand{\wt}{\widetilde}

\newcommand{\w}{\wedge}

\makeatletter
\let\c@equation\c@thm
\makeatother
\numberwithin{equation}{section}

\begin{document}

\title{Observations about the Lie algebra $\lieg \subset \so(7)$, \\
associative $3$-planes, and $\so(4)$ subalgebras}

\author{Max Chemtov \\ \emph{Department of Mathematics and Statistics, McGill University} \\ \tt{max.chemtov@mail.mcgill.ca} \and Spiro Karigiannis \\ \emph{Department of Pure Mathematics, University of Waterloo} \\ \tt{karigiannis@uwaterloo.ca}}

\maketitle

\begin{abstract}
We make several observations relating the Lie algebra $\lieg \subset \so(7)$, associative $3$-planes, and $\so(4)$ subalgebras. Some are likely well-known but not easy to find in the literature, while other results are new. We show that an element $X \in \lieg$ cannot have rank $2$, and if it has rank $4$ then its kernel is an associative subspace. We prove a canonical form theorem for elements of $\lieg$. Given an associative $3$-plane $P$ in $\R^7$, we construct a Lie subalgebra $\Theta(P)$ of $\so(7) = \Lambda^2 (\R^7)$ that is isomorphic to $\so(4)$. This $\so(4)$ subalgebra differs from other known constructions of $\so(4)$ subalgebras of $\so(7)$ determined by an associative $3$-plane. These are results of an NSERC undergraduate research project. The paper is written so as to be accessible to a wide audience.
\end{abstract}

\tableofcontents

\section{Introduction} \label{introsec}

The space $\R^7$ admits a special algebraic structure called a ``$\G$-structure'', which essentially consists of the standard inner product, orientation, and a \emph{cross product} operation $(v, w) \mapsto v \times w$ satisfying certain properties. The subgroup of $\GL{7, \R}$ preserving this structure is the exceptional Lie group $\G$, embedded in a natural way inside $\SO{7}$. Since $\so(7) = \Lambda^2 (\R^7)$, this yields a decomposition $\Lambda^2 (\R^7) = \lieg \oplus \lieg^{\perp}$. Moreover, the $\G$-structure on $\R^7$ determines a distinguished class of $3$-dimensional subspaces of $\R^7$, called \emph{associative} subspaces. In this paper we investigate the relations between associative subspaces, the Lie algebra $\lieg$, and certain $\so(4)$ subalgebras of $\Lambda^2 (\R^7)$.

Some of the results in Section~\ref{sec:g2} are likely well-known to experts, but are not easy to find explicitly in the literature. For example, we show that an element $X$ of $\lieg$ can never have rank $2$, and if it has rank $4$ then its kernel is an associative subspace. We also derive a general canonical form theorem (Theorem~\ref{thm:g2-canonical-form}) for elements of $\lieg$, which is essentially the maximal torus theorem, at the Lie algebra level, for $\lieg$. However, we give a completely self-contained proof using the cross product and its properties, avoiding the use of any abstract Lie theory. See Remark~\ref{rmk:maximal-torus} for more details. This canonical form theorem has been recently applied to study Betti numbers of nearly $\G$ manifolds. See Remark~\ref{rmk:anton} for more details.

The results of Section~\ref{sec:assoc-so4} appear to be new. Given an associative subspace $P$ of $\R^7$, we construct an $\so(4)$ subalgebra $\Theta(P)$ of $\Lambda^2 (\R^7)$. We also investigate the intersections $\Theta(P) \cap \Theta(Q)$ for two distinct associative subspaces $P, Q$. These results are summarized in Theorem~\ref{thm:so4-final}. The $\so(4)$ subalgebra $\Theta(P)$ determined by an associative subspace $P$ constructed in the present paper is \emph{not the same} as the standard embedding of $\so(4)$ in $\lieg$ determined by $P$, because $\Theta(P)$ \emph{does not} lie in $\lieg$. It is also \emph{not the same} as the $\so(4)$ subalgebra obtained from $\Lambda^2 (P^{\perp})$. See Remark~\ref{rmk:different-so4} for more details.

\textbf{Notation and conventions.} We equip $\R^7$ with the standard inner product $\langle \cdot, \cdot \rangle$ and orientation. Using the inner product, we identify $\R^7$ with its dual space in the usual way throughout without further mention. Thus all of our indices are subscripts. In particular, a bilinear form $A$ on $\R^7$ is identified with a linear operator $A \colon \R^7 \to \R^7$ by $A(u,v) = \langle A(u), v \rangle$. Under this identification a symmetric bilinear form (respectively, skew-symmetric bilinear form) corresponds to a self-adjoint operator (respectively, skew-adjoint operator). For example, if $u, v \in \R^7$, then the $2$-form $u \w v$ corresponds to the skew-adjoint linear map
\begin{equation} \label{eq:2form-map}
(u \w v) w = \langle u, w \rangle v - \langle v, w \rangle u.
\end{equation}

We only work with orthonormal bases $\{ e_1, \ldots, e_7 \}$, so the matrix representing a bilinear form is the same as the matrix representing the associated linear operator, namely $A_{ij} = A(e_i, e_j) = \langle A(e_i), e_j \rangle$. In index notation with respect to an orthonormal frame, we sum over repeated indices from $1$ to $7$. We write $\Lambda^k$ for $\Lambda^k (\R^7)$, the space of $k$-forms.

\textbf{Acknowledgements.} A significant portion of this research was conducted in summer 2021, while the first author held an NSERC Undergraduate Student Research Award under the supervision of the second author. The work was completed in 2022 by the second author, under the support of an NSERC Discovery Grant. The second author also thanks Shengda Hu and Anton Iliashenko for useful conversations. Both authors are grateful to the anonymous referee for finding typographical errors in the original draft.

\section{Preliminaries} \label{sec:preliminaries}

In this section we review the standard ``$\G$ package'' on $\R^7$, including the notion of associative subspaces, and discuss without proofs their basic properties. A good reference for some of this material, with similar conventions and notation, is~\cite{K-intro}. Another useful resource is~\cite{SW}. A representation-theoretical presentation is given in~\cite{Bryant}, and a nice historical survey of $\G$ is given in~\cite{A}. For more details on associative subspaces, see~\cite{Harvey, HL, Leung, Lotay}. However, as long as the reader is willing to accept the facts presented in this section, then the rest of the paper is essentially completely self-contained.

\subsection{Review of the standard $\G$-structure on $\R^7$} \label{sec:review}

Recall that the \emph{octonions} $\Oc$ are an $8$-dimensional non-associative real normed unital algebra. As a vector space, $\Oc = \R^8$, equipped with its standard inner product and orientation. As a normed unital algebra, $\Oc$ is equipped with a bilinear (non-associative) multiplication $(a, b) \mapsto ab$ satisfying $| ab | = |a| \, |b|$ where $| \cdot |$ is the norm induced from the inner product. The orthogonal complement of the span of the multiplicative identity $1$ is called the \emph{imaginary} octonions and is denoted $\imag \Oc$. Thus $\imag \Oc = \R^7$ inherits the Euclidean inner product and we give it the induced orientation such that $1 \w \vol_{\imag \Oc} = \vol_{\Oc}$.

The octonion multiplication induces a \emph{cross product} operation $\times \colon \R^7 \times \R^7 \to \R^7$ by
$$ u \times v = \imag (uv), \qquad \text{for $u, v \in \R^7 = \imag \Oc$.} $$
The cross product satisfies
$$ u \times v = - v \times u, \qquad \langle u \times v, u \rangle = 0, \qquad | u \times v |^2 = | u \w v |^2. $$
From the cross product, we obtain a $3$-form $\ph \in \Lambda^3$ given by
\begin{equation} \label{eq:ph-cross}
\ph(u, v, w) = \langle u \times v, w \rangle.
\end{equation}
From~\eqref{eq:ph-cross}, we deduce immediately that with respect to an orthonormal frame we have
\begin{equation} \label{eq:cross-frame}
(u \times v)_k = u_p v_q \ph_{pqk} = (v \hk u \hk \ph)_k.
\end{equation}
The Hodge star operator $\sta \colon \Lambda^k \to \Lambda^{7-k}$ then gives us a $4$-form $\ps \in \Lambda^4$ by
$$ \ps = \sta \ph. $$
Using the inner product, $\ps$ can also be thought of as a vector-valued $3$-form by
$$ \langle \ps(u, v, w), y \rangle = \ps(u,v,w,y). $$
That is, we ``raise the last index'' of $\ps$. (This vector valued $3$-form is denoted $\chi$ by some authors. We choose to call it $\ps$ to avoid proliferation of notation. Similarly we could write $u \times v = \ph(u,v)$ if we raise the last index of $\ph$ to regard it as a vector-valued $2$-form.)

Iterating the cross product, one obtains the relation
\begin{equation} \label{eq:iterated-cross}
u \times (v \times w) = - \langle u, v \rangle w + \langle u, w \rangle v + \ps(u, v, w).
\end{equation}
A special case occurs when $v = u$, for which we have
\begin{equation} \label{eq:iterated-cross-same}
u \times (u \times w) = - |u|^2 w + \langle u, w \rangle u.
\end{equation}
Moreover, it follows from~\eqref{eq:iterated-cross} and the skew-symmetry of $\ps$ that
\begin{equation} \label{eq:iterated-cross-orthog}
u \times (v \times w) \qquad \text{is totally skew-symmetric if $u, v, w$ are orthogonal.}
\end{equation}
In index notation with respect to an orthonormal frame, the fundamental identity~\eqref{eq:iterated-cross} is
\begin{equation} \label{eq:fundamental}
\ph_{ijp} \ph_{abp} = \delta_{ia} \delta_{jb} - \delta_{ib} \delta_{ja} - \ps_{ijab}.
\end{equation}
By taking various contractions of this identity with itself, one can establish (see~\cite{K-flows} for proofs) the following important identities for contractions of $\ph$ with $\ps$:
\begin{equation}
\begin{aligned} \label{eq:phps}
\ph_{ijp} \ps_{abcp} & = \delta_{ia} \ph_{jbc} + \delta_{ib} \ph_{ajc} + \delta_{ic} \ph_{abj} - \delta_{ja} \ph_{ibc} - \delta_{jb} \ph_{aic} - \delta_{jc} \ph_{abi}, \\
\ph_{ipq} \ps_{abpq} & = - 4 \ph_{iab},
\end{aligned}
\end{equation}
and for contractions of $\ps$ with $\ps$:
\begin{equation}
\begin{aligned} \label{eq:psps}
\ps_{ijkp} \ps_{abcp} & = - \ph_{ajk} \ph_{ibc} - \ph_{iak} \ph_{jbc} - \ph_{ija} \ph_{kbc} \\
& \qquad {} + \delta_{ia} \delta_{jb} \delta_{kc} + \delta_{ib} \delta_{jc} \delta_{ka} + \delta_{ic} \delta_{ja} \delta_{kb} - \delta_{ia} \delta_{jc} \delta_{kb} - \delta_{ib} \delta_{ja} \delta_{kc} - \delta_{ic} \delta_{jb} \delta_{ka} \\
& \qquad {} - \delta_{ia} \ps_{jkbc} - \delta_{ja} \ps_{kibc} - \delta_{ka} \ps_{ijbc} + \delta_{ab} \ps_{ijkc} - \delta_{ac} \ps_{ijkb}, \\
\ps_{ijpq} \ps_{abpq} & = 4 \delta_{ia} \delta_{jb} - 4 \delta_{ib} \delta_{ja} - 2 \ps_{ijab}.
\end{aligned}
\end{equation}

The $\G$-structure $\ph$ on $\R^7$ determines distinguished classes of $3$-dimensional subspaces.

\begin{defn} \label{defn:assoc}
A $3$-dimensional subspace $A$ of $\R^7$ is called \emph{associative} if it is closed under the cross product. That is, if $u, v \in A$ implies $u \times v \in A$.
\end{defn}

Using the properties of the cross product, we see easily that if $\{ e_1, e_2, e_3 \}$ is an orthonormal basis of $A$, then $A$ is associative if and only if $e_i \times e_j = \pm e_k$ for $i, j, k$ distinct. Equivalently, $A$ is associative if and only if $A$ admits an orientation such that $\rest{\ph}{A} = \vol_A$. In this case, if $\{ e_1, e_2, e_3 \}$ is an oriented orthonormal basis of $A$, then $e_i \times e_j = e_k$ whenever $i,j,k$ is a cyclic permutation of $1,2,3$. It is also true that
\begin{equation} \label{eq:assoc-character}
\text{$A$ is associative if and only if $\ps(u,v,w) = 0$ for all $u, v, w \in A$.}
\end{equation}
(This can be seen from~\eqref{eq:iterated-cross} by letting $u,v,w$ be orthonormal and using the fact that $u \times y = 0$ if and only if $u, y$ are parallel.)

The exceptional Lie group $\G$ can be defined as the stabilizer in $\GL{7, \R}$ of the $3$-form $\ph$. It is a classical fact that, from this definition, it follows that $\G \subseteq \SO{7}$. (See for example~\cite[Theorem 4.4]{K-intro} for a proof.) In particular, this means that elements of $\G$ preserve the inner product and the orientation on $\R^7$. The columns of a matrix $P \in \SO{7}$ are an oriented orthonormal frame $\{ e_1, \ldots, e_7 \}$ of $\R^7$. Those matrices $P \in \G \subseteq \SO{7}$ are further distinguished by the relations
\begin{equation} \label{eq:G2-adapted}
e_3 = e_1 \times e_2, \qquad e_5 = e_1 \times e_4, \qquad e_6 = e_2 \times e_4, \qquad e_7 = e_3 \times e_4 = (e_1 \times e_2) \times e_4.
\end{equation}
An oriented orthonormal frame $\{e_1, \ldots, e_7\}$ for $\R^7$ is called a \emph{$\G$-adapted frame} if it satisfies~\eqref{eq:G2-adapted}. Such frames are precisely the images of the standard basis of $\R^7$ by elements of $\G$.

From the octonion multiplication table and the definitions of $\ph$ and $\ps$, one can show that with respect to a $\G$-adapted frame $\{ e_1, \ldots, e_7 \}$, and writing $e_{ijk} = e_i \w e_j \w e_k$, we have
\begin{equation} \label{eq:forms-coordinates}
\begin{aligned}
\ph & = e_{123} - e_{167} - e_{527} - e_{563} - e_{415} - e_{426} - e_{437}, \\
\ps & = e_{4567} - e_{4523} - e_{4163} - e_{4127} - e_{2637} - e_{1537} - e_{1526}.
\end{aligned}
\end{equation}
From~\eqref{eq:forms-coordinates}, it is easy to read off, for example, that $e_1 \times e_2 = e_3$ and $e_5 \times e_6 = - e_3$. Thus $\{ e_1, e_2, e_3 \}$ and $\{ -e_3, e_5, e_6 \}$ are oriented orthonormal bases of associative $3$-planes.

\subsection{The decomposition $\Lambda^2 = \Lambda^2_7 \oplus \Lambda^2_{14}$} \label{sec:lambda-2}

Using the inner product on $\R^7$, the Lie algebra $\so(7)$ of skew-adjoint operators is identified with the space $\Lambda^2$ of skew-symmetric bilinear forms. There is an orthogonal decomposition $\Lambda^2 = \Lambda^2_7 \oplus \Lambda^2_{14}$ into irreducible representations of $\lieg$ that is described as follows. (See~\cite{K-intro} or~\cite{K-flows} for more details about those facts that are presented here without proof.)

Since $\G \subseteq \SO{7}$, we get $\lieg \subseteq \so(7)$, and under the identification $\Lambda^2 = \so(7)$ we have $\Lambda^2_{14} = \lieg$. Since $\G$ is the subgroup of $\SO{7}$ preserving $\ph$, the infinitesimal expression of this condition is
\begin{equation} \label{eq:X-in-g2-variant}
X \in \Lambda^2_{14} \quad \iff \quad \ph(X(u), v, w) + \ph(u, X(v), w) + \ph(u, v, X(w)) = 0 \quad \text{ for all $u, v, w$.}
\end{equation}
Using~\eqref{eq:ph-cross}, we have $\ph(X(u), v, w) = \langle X(u), v \times w \rangle = X(u, v \times w)$, so we can rewrite~\eqref{eq:X-in-g2-variant} as
\begin{equation} \label{eq:X-in-g2}
X \in \Lambda^2_{14} \quad \iff \quad  X(u, v \times w) + X(v, w \times u) + X(w, u \times v) = 0 \quad \text{ for all $u, v, w$.}
\end{equation}
Note that $\Lambda^2_{14}$ is a Lie subalgebra of $\Lambda^2$. In terms of a local orthonormal frame, $X \in \Lambda^2_{14}$ if and only if
\begin{equation} \label{eq:g2-deriv}
X_{ip} \ph_{pjk} + X_{jp} \ph_{ipk} + X_{kp} \ph_{ijp} = 0.
\end{equation}
We can also express~\eqref{eq:X-in-g2} in the following useful equivalent form. Using the skew-symmetry of $X(u,v) = \langle X(u), v \rangle$ and of $\ph(u,v,w) = \langle u \times v, w \rangle$, we have
\begin{align*}
X(u, v \times w) + X(v, w \times u) + X(w, u \times v) & = - X(v \times w, u) + \langle X(v), w \times u \rangle + \langle X(w), u \times v \rangle \\
& = - \langle X(v \times w), u \rangle + \langle u, X(v) \times w \rangle + \langle u, v \times X(w) \rangle \\
& = \langle - X(v \times w) + X(v) \times w + v \times X(w), u \rangle.
\end{align*}
By the nondegeneracy of the inner product, we deduce from the above that
\begin{equation} \label{eq:X-in-g2-alt}
X \in \Lambda^2_{14} \quad \iff \quad  X(v \times w) = X(v) \times w + v \times X(w) \quad \text{ for all $v, w$.}
\end{equation}
That is,~\eqref{eq:X-in-g2-alt} says that $\lieg = \Lambda^2_{14} \subseteq \Lambda^2 = \so(7)$ consists of those derivations in $\so(7)$ that (infinitesimally) preserve the cross product $\times$. This is exactly what we expect, because $g$ and $\times$ together determine $\ph$, and $\lieg$ is precisely the space of derivations in $\mathfrak{gl}(7, \R)$ preserving $\ph$.

The following observation about elements of $\Lambda^2_{14}$ is used twice in Section~\ref{sec:g2}.
\begin{prop} \label{prop:X-in-g2-identity}
Let $X \in \Lambda^2_{14} = \lieg$, and suppose that $u, v, y$ are \emph{orthogonal}. Then we have
$$ X(u \times y, v \times y) = |y|^2 X(u, v) + X(y, (u \times v) \times y). $$
\end{prop}
\begin{proof}
Using~\eqref{eq:X-in-g2},~\eqref{eq:iterated-cross-same}, and~\eqref{eq:iterated-cross-orthog}, we compute
\begin{align*}
X(u \times y, v \times y) & = - X(v, y \times (u \times y)) - X(y, (u \times y) \times v) \\
& = X(v, y \times (y \times u) ) + X(y, v \times (u \times y) ) \\
& = X(v, - |y|^2 u) + X(y, y \times (v \times u)) \\
& = |y|^2 X(u,v) + X(y, (u \times v) \times y)
\end{align*}
as claimed.
\end{proof}

The subspace $\Lambda^2_7$, which is \emph{not} a Lie subalgebra (see Section~\ref{sec:structure-splitting}), is the orthogonal complement of $\Lambda^2_{14}$. In fact we have (see~\cite{K-intro, K-flows}) that
\begin{equation} \label{eq:27-first}
X \in \Lambda^2_7 \quad \iff \quad X = u \hk \ph \quad \text{for some $u \in \R^7$}
\end{equation}
and indeed the map $\R^7 \to \Lambda^2_7$ given by $u \mapsto u \hk \ph$ is a linear isomorphism.

There are various equivalent descriptions of the spaces $\Lambda^2_7$ and $\Lambda^2_{14}$ that are also useful. These are
\begin{equation} \label{eq:Lambda-2-7-14}
\begin{aligned}
X & \in \Lambda^2_7 & \iff \quad \sta (X \w \ph) & = -2 X & \iff \quad X = u \hk \ph, \\
X & \in \Lambda^2_{14} & \iff \quad \sta (X \w \ph) & = X & \iff \quad X \w \ps = 0.
\end{aligned}
\end{equation}
From the fact that $(\sta (X \w \ph))_{ij} = \frac{1}{2} \ps_{ijpq} X_{pq}$, the above descriptions in terms of a local orthonormal frame become
\begin{equation} \label{eq:Lambda-2-7-14-frame}
\begin{aligned}
X_{ij} & \in \Lambda^2_7 & \iff \quad \ps_{ijpq} X_{pq} & = -4 X_{ij} & \iff \quad X_{ij} = u_p \ph_{pij}, \\
X_{ij} & \in \Lambda^2_{14} & \iff \quad \ps_{ijpq} X_{pq} & = 2 X_{ij} & \iff \quad X_{pq} \ph_{pqk} = 0,
\end{aligned}
\end{equation}
and moreover we have
\begin{equation} \label{eq:Pbeta}
\ps_{ijpq} X_{pq} = - 4 (X_7)_{ij} + 2 (X_{14})_{ij}
\end{equation}
where $X_7$ and $X_{14}$ denote, respectively, the orthogonal projections of $X \in \Lambda^2$ onto $\Lambda^2_7$ and $\Lambda^2_{14}$. We use~\eqref{eq:Pbeta} in Section~\ref{sec:structure-splitting} to study the splitting $\Lambda^2 = \Lambda^2_7 \oplus \Lambda^2_{14}$ at a deeper level.

Finally, there is a useful formula for the projection $X_7$ of an element $X \in \Lambda^2$, given by
\begin{equation} \label{eq:X7-projection}
(X_7)_{ij} = u_p \ph_{pij} \qquad \text{where} \qquad u_p = \tfrac{1}{6} X_{ij} \ph_{ijp}.
\end{equation}
Consider the special case $X = v \w w \in \Lambda^2$. Then $X_{ij} = v_i w_j - v_j w_i$, so equations~\eqref{eq:X7-projection} and~\eqref{eq:cross-frame} give $(\pi_7 (u \w v))_{ij} = u_p \ph_{pij}$ where $u_p = \frac{1}{6} (v_i w_j - v_j w_i) \ph_{ijp} = \frac{1}{3} v_i w_j \ph_{ijp} = \frac{1}{3} (v \times w)_p$. That is, we have shown that
\begin{equation} \label{eq:pi7-vw}
\pi_7 (v \w w) = \tfrac{1}{3} (v \times w) \hk \ph.
\end{equation}

\subsection{Canonical forms for $\so(7) = \Lambda^2$ and for $\Lambda^2_7$} \label{sec:so7-canonical-form}

Let us recall the canonical form theorem for a real skew-symmetric bilinear form on $\R^n$ (specialized to the case $n=7)$. If $X \in \Lambda^2$, the matrix for $X$ with respect to the standard basis is skew-symmetric, and thus the linear operator $X \colon \R^7 \to \R^7$ given by $\langle X(u), v \rangle = X(u, v)$ is skew-adjoint. (In particular, its matrix with respect to \emph{any} orthonormal frame for $\R^7$ will be skew-symmetric.)

It is then a standard result from linear algebra that the eigenvalues of $X$ are all purely imaginary and come in complex conjugate pairs. That is, there exist real numbers
$$ \lambda \geq \nu \geq \mu \geq 0 $$
such that the eigenvalues of $X$ are $\pm i \lambda$, $\pm i \nu$, $\pm i \mu$, and $0$. It follows that we can \emph{orthogonally} decompose $\R^7 = E_{\lambda} \oplus E_{\nu} \oplus E_{\mu} \oplus E_0$ into $X$-invariant subspaces where:
\begin{itemize}
\item The space $E_0$ is a $1$-dimensional subspace, on which $X$ vanishes (this space corresponds to the one eigenvalue that is necessarily zero, although there could be more as $\lambda \geq \nu \geq \mu \geq 0$).
\item The space $E_{\lambda}$ is a $2$-dimensional subspace for which we can find an orthonormal basis $\{ u, \wt u \}$ such that $X(u) = \lambda \wt u$ and $X(\wt u) = - \lambda u$. That is, restricted to $E_{\lambda}$, the operator $X$ is a $90^{\circ}$ degree counterclockwise rotation of the oriented plane $u \w \wt u$ followed by a dilation by $\lambda \geq 0$. (Analogously for the spaces $E_{\nu}$ and $E_{\mu}$). In particular, note that we can replace the ordered basis $\{ u, \wt u \}$ by $\{ u', \wt u' \} = \{ (\cos \alpha) u - (\sin \alpha) \wt u, (\sin \alpha) u + (\cos \alpha) \wt u \}$ and we still have $X(u') = \lambda \wt u'$, $X(\wt u') = - \lambda u'$.
\end{itemize}
In summary, there exists an orthonormal basis $\cB = \{ e_1, \ldots, e_7 \}$ for $\R^7$ such that, with respect to the basis $\cB$, the matrix $[X]_{\cB}$ for $X$ takes the form
\begin{equation} \label{eq:so7-canonical-form}
[X]_{\cB} = \begin{pmatrix} 0 & & & \\ & \begin{matrix} 0 & -\mu \\ \mu & 0 \end{matrix} & & \\ & & \begin{matrix} 0 & -\nu \\ \nu & 0 \end{matrix} & \\ & & & \begin{matrix} 0 & -\lambda \\ \lambda & 0 \end{matrix} \end{pmatrix}.
\end{equation}
In particular, the rank of $X$ must be one of $0$, $2$, $4$, or $6$.

We use the above canonical form to prove a refined canonical form theorem for elements of $\Lambda^2_{14} = \lieg$ in Theorem~\ref{thm:g2-canonical-form}.

Here we consider the easier case $\Lambda^2_7$. Let $X = u \hk \ph \in \Lambda^2_7$. From~\eqref{eq:ph-cross} we have
$$ \langle X(v), w \rangle = X(v, w) = (u \hk \ph)(v, w) = \ph(u,v,w) = \langle u \times v, w \rangle. $$
Thus we have $X(v) = u \times v$ for all $v$. Write $u = \lambda e_1$, where $\lambda = | u | \geq 0$. If $X \neq 0$, then $u \neq 0$ and $\lambda > 0$. In this case $e_1$ is uniquely determined by $e_1 = \lambda^{-1} u$. If $u = 0$, then $\lambda = 0$ and $e_1$ is arbitrary. Note that $X = \lambda Y$, where $Y(v) = e_1 \times v$.

Complete $\{ e_1 \}$ to a $\G$-adapted frame $\cB = \{ e_1, \ldots, e_7 \}$. Then from~\eqref{eq:forms-coordinates} we have
\begin{align*}
Y(e_1) & = e_1 \times e_1 = 0, & \quad Y(e_2) & = e_1 \times e_2 = e_3, & \quad Y(e_3) & = e_1 \times e_3 = - e_2, \\
Y(e_4) & = e_1 \times e_4 = e_5, & \quad Y(e_5) & = e_1 \times e_5 = - e_4, & \quad Y(e_6) & = e_1 \times e_6 = - e_7, & \quad Y(e_7) & = e_1 \times e_7 = e_6.
\end{align*}
The above equations express the fact that the linear map $Y$ maps the orthogonal complement of the span of $\{ e_1 \}$ into itself, and squares to minus the identity. This can also be seen from~\eqref{eq:iterated-cross-same}.

With respect to the frame $\cB$, the matrix $[X]_{\cB}$ for $X$ takes the form
\begin{equation} \label{eq:canonical-form-7}
[X]_{\cB} = \begin{pmatrix} 0 & & & \\ & \begin{matrix} 0 & -\lambda \\ \lambda & 0 \end{matrix} & & \\ & & \begin{matrix} 0 & -\lambda \\ \lambda & 0 \end{matrix} & \\ & & & \begin{matrix} 0 & \lambda \\ -\lambda & 0 \end{matrix} \end{pmatrix}.
\end{equation}

\begin{cor} \label{cor:canonical-form-7}
Let $X \in \Lambda^2_7$. Then $\rank X \in \{ 0, 6 \}$. That is, $\rank X \neq 2, 4$.
\end{cor}
\begin{proof}
If $X \neq 0$, then $\lambda > 0$, so $\rank X = 6$ by~\eqref{eq:canonical-form-7}.
\end{proof}

\begin{rmk} \label{rmk:canonical-form-7}
Note that~\eqref{eq:canonical-form-7} is \emph{almost} of the form~\eqref{eq:so7-canonical-form}, with $\mu = \nu = \lambda$, except that we need to replace $e_6 \mapsto - e_6$. That is, one cannot achieve the canonical form~\eqref{eq:so7-canonical-form} for an element $X \in \Lambda^2_7$ using a $\G$-adapted frame, but rather with a frame that is obtained from a $\G$-adapted frame by a reflection. The fact that $\lambda = \mu = \nu$ is what rules out the possibilities of rank $2$ or $4$. (The reader should compare this with Remark~\ref{rmk:g2-canonical-form} below, where we observe that one \emph{can} achieve the canonical form~\eqref{eq:so7-canonical-form} for an element $X \in \Lambda^2_{14}$ using a $\G$-adapted frame.)
\end{rmk}

\section{The Lie algebra $\lieg = \Lambda^2_{14}$ as a subalgebra of $\so(7) = \Lambda^2$} \label{sec:g2}

In this section we consider the inclusion of Lie algebras $\lieg \subseteq \so(7)$. We first establish some facts about elements of $\Lambda^2_{14}$, namely that: (i) if their rank is less than maximal, their kernel contains an associative $3$-plane; and (ii) elements of $\Lambda^2_{14}$ cannot have rank $2$. Then we prove a general canonical form theorem for elements of $\Lambda^2_{14}$. See also Remark~\ref{rmk:maximal-torus}.

\subsection{The kernel and rank of elements of $\lieg = \Lambda^2_{14}$} \label{sec:rank-kernel-g2}

Let $X \in \Lambda^2_{14}$. Since $X$ is identified with a skew-adjoint operator, we have $X^* = - X$, and thus $\im X = (\ker X^*)^{\perp} = (\ker X)^{\perp}$, so we have an orthogonal decomposition $\R^7 = (\ker X) \oplus (\im X)$.

\begin{prop} \label{prop:kernel-assoc}
Let $X \in \Lambda^2_{14}$. Suppose $\rank X \leq 4$. Then $\ker X$ contains an associative $3$-plane.
\end{prop}
\begin{proof}
Since $\rank X \leq 4$ and $\R^7 = (\ker X) \oplus (\im X)$, we have $\dim (\ker X) \geq 3$. Thus there exists an orthonormal pair $u_1, u_2 \in \ker X$. Let $w \in (\ker X)^{\perp} = \im X$. There exists $v \in \R^7$ with $w = X(v)$. Since $X(u_1) = X(u_2) = 0$, we have
\begin{align*}
\langle u_1 \times u_2, w \rangle & = \ph(u_1, u_2, w) = \ph(u_1, u_2, X(v)) \\
& = \ph(X(u_1), u_2, v) + \ph(u_1, X(u_2), v) + \ph(u_1, u_2, X(v)),
\end{align*}
which vanishes by~\eqref{eq:X-in-g2-variant}, so $u_1 \times u_2$ is orthogonal to $w$. Since $w \in (\ker X)^{\perp}$ is arbitrary, we find that $u_1 \times u_2 \in \ker X$, so $\ker X$ contains the associative $3$-plane spanned by $\{ u_1, u_2, u_1 \times u_2 \}$.
\end{proof}

\begin{cor} \label{cor:kernel-assoc}
Let $X \in \Lambda^2_{14}$. Suppose $\rank X \leq 4$. Then there exists a $\G$-adapted frame $\cB = \{ e_1, \ldots, e_7 \}$ such that, with respect to the basis $\cB$, the matrix $[X]_{\cB}$ takes the form
\begin{equation} \label{eq:kernel-assoc-temp1}
[X]_{\cB} = \begin{pmatrix} 0_{3 \times 3} & 0_{3 \times 4} \\ 0_{4 \times 3} & Y_{4 \times 4} \end{pmatrix}
\end{equation}
where the $4 \times 4$ matrix $Y \in \so(4)$ is of the form
\begin{equation} \label{eq:kernel-assoc-temp2}
Y = \begin{pmatrix} 0 & - a & - b & - c \\ a & 0 & -c & b \\ b & c & 0 & -a \\ c & -b & a & 0 \end{pmatrix}
\end{equation}
for some $a, b, c \in \R$.
\end{cor}
\begin{proof}
By Proposition~\ref{prop:kernel-assoc}, we can find $\{ e_1, e_2, e_3 = e_1 \times e_2 \}$ orthonormal in $\ker X$. Let $e_4$ be a unit vector orthogonal to $\{ e_1, e_2, e_3 \}$, and define the $\G$-adapted frame $\cB = \{ e_1, \ldots, e_7 \}$ by~\eqref{eq:G2-adapted}. Since $X(e_i, v) = \langle X(e_i), v \rangle = 0$ for $i = 1, 2, 3$ and any $v \in \R^7$, we see that $X$ has the block diagonal form~\eqref{eq:kernel-assoc-temp2}.

Letting $y = e_4$ and $u, v$ be distinct elements of $\{ e_1, e_2, e_3 \}$ in Proposition~\ref{prop:X-in-g2-identity}, we obtain
$$ X( e_i \times e_4, e_j \times e_4 ) = X(e_i, e_j) + X(e_4, (e_i \times e_j) \times e_4) = X(e_i, e_j) + X(e_4, e_k \times e_4). $$
The above yields the three independent identities
$$ X_{56} =  X_{47}, \qquad X_{67} = X_{45}, \qquad X_{75} = X_{46}. $$
These conditions and the skew-symmetry of $X$ give~\eqref{eq:kernel-assoc-temp2}.
\end{proof}

\begin{cor} \label{cor:no-rank-2}
Let $X \in \Lambda^2_{14}$. Then $\rank X \in \{ 0, 4, 6 \}$. That is, $\rank X \neq 2$.
\end{cor}
\begin{proof}
We already know that $\rank X$ is even. Suppose $\rank X \leq 4$. By Corollary~\ref{cor:kernel-assoc}, there exists a basis in which the matrix for $X$ takes the form~\eqref{eq:kernel-assoc-temp1}, where $Y$ is given by~\eqref{eq:kernel-assoc-temp2}. A computation reveals that $\det Y = (a^2 + b^2 + c^2)^2$, and thus if $Y \neq 0$, then $\rank Y = 4$. Hence $\rank X \in \{ 0, 4, 6 \}$.
\end{proof}

Of course, since $\rank X \geq 6$ is an open condition, the generic element of $\Lambda^2_{14}$ will have rank $6$.

\begin{cor} \label{cor:ranks}
Let $X \in \Lambda^2$. If $\rank X = 2$, then $X$ cannot be purely type $\Lambda^2_7$ nor purely type $\Lambda^2_{14}$. If $\rank X = 4$, then $X$ cannot be purely type $\Lambda^2_7$.
\end{cor}
\begin{proof}
This is immediate from Corollaries~\ref{cor:canonical-form-7} and~\ref{cor:no-rank-2}.
\end{proof}

\subsection{A canonical form theorem for elements of $\lieg = \Lambda^2_{14}$} \label{sec:canonical-form}

In this section we extend the results of Section~\ref{sec:rank-kernel-g2} by removing the restriction that $\rank X \leq 4$ for $X \in \Lambda^2_{14}$. In fact, this section is logically independent of Section~\ref{sec:rank-kernel-g2}, but it is instructive to have considered the special case $\rank X \leq 4$ before tackling the general case. The following is a canonical form theorem for elements of $\Lambda^2_{14} = \lieg$.

\begin{thm} \label{thm:g2-canonical-form}
Let $X \in \Lambda^2_{14} = \lieg$. There exists a $\G$-adapted frame $\cB = \{ e_1, \ldots, e_7 \}$ and real numbers $\lambda \geq \mu \geq 0$ such that, with respect to the basis $\cB$, the matrix $[X]_{\cB}$ takes the form
\begin{equation} \label{eq:g2-canonical-form-exp1}
[X]_{\cB} = \begin{pmatrix} \begin{matrix} 0 & - \lambda & 0 \\ \lambda & 0 & 0 \\ 0 & 0 & 0 \end{matrix} & \\ & \begin{matrix} 0 & 0 & 0 & \lambda - \mu \\ 0 & 0 & - \mu & 0 \\ 0 & \mu & 0 & 0 \\ - (\lambda - \mu) & 0 & 0 & 0 \end{matrix} \end{pmatrix}
\end{equation}
where all nondisplayed entries are zero. Moreover, there are exactly three cases that occur:
\begin{enumerate}[(i)]
\item if $\lambda = \mu$, then $\lambda = \mu = 0$, and $X = 0$;
\item if $\lambda > 0$ and $\mu = 0$, then $\rank X = 4$, and $\ker X$ is the associative $3$-plane spanned by $\{ e_3, e_5, e_6 \}$;
\item if $\lambda > \mu > 0$, then $\rank X = 6$, and $\ker X$ is spanned by $\{ e_3 \}$.
\end{enumerate}
\end{thm}
\begin{proof}
Let $X \in \Lambda^2_{14} \subseteq \Lambda^2$. Let $\lambda \geq \nu \geq \mu \geq 0$ be as in Section~\ref{sec:so7-canonical-form}. If $\lambda = 0$, then $X = 0$ and the theorem holds trivially. Thus we can assume from now on that $\lambda > 0$.

There exists an orthonormal basis $\{ e_1, e_2 \}$ for $E_{\lambda}$ such that $X(e_1) = \lambda e_2$ and $X(e_2) = - \lambda e_1$. Let $e_3 = e_1 \times e_2$, so that $\{ e_1, e_2, e_3 \}$ spans an associative $3$-plane $A$. We have
\begin{align*}
X(e_1, e_3) & = \langle X(e_1), e_3 \rangle = \lambda \langle e_2, e_1 \times e_2 \rangle = 0, \\
X(e_2, e_3) & = \langle X(e_2), e_3 \rangle = - \lambda \langle e_1, e_1 \times e_2 \rangle = 0.
\end{align*}
Let $C = A^{\perp}$ be the orthogonal coassocative $4$-plane. Let $v \in C$. Then, using~\eqref{eq:X-in-g2}, we have
\begin{align*}
X(v, e_3) & = X(v, e_1 \times e_2) \\
& = - X(e_1, e_2 \times v) - X(e_2, v \times e_1) \\
& = - \langle X(e_1), e_2 \times v \rangle - \langle X(e_2), v \times e_1 \rangle \\
& = - \lambda \langle e_2, e_2 \times v \rangle + \lambda \langle e_1, v \times e_1 \rangle = 0.
\end{align*}
We have thus shown that $X(v, e_3) = 0$ for any $v$ orthogonal to $e_3$. In particular, if $\cB_{C} = \{ e_4, e_5, e_6, e_7 \}$ is any orthonormal basis for $C$, then with respect to $\cB = \{ e_1, \ldots, e_7 \}$, we have
\begin{equation} \label{eq:g2-canonical-form-temp}
[X]_{\cB} = \begin{pmatrix} \begin{matrix} 0 & - \lambda & 0 \\ \lambda & 0 & 0 \\ 0 & 0 & 0 \end{matrix} & \\ & [\rest{X}{C}]_{{\cB}_C} \end{pmatrix}
\end{equation}
where the skew-symmetric $4 \times 4$ matrix $[\rest{X}{C}]_{{\cB}_C}$ represents the restriction $\rest{X}{C}$ with respect to $\cB_C$. (In passing, we note that if $\mu > 0$, then $e_3$ must lie in the space $E_0$, but this need not hold if $\mu = 0$.)

Note that if we restrict $X$ to $A$, it has characteristic polynomial $t (t^2 + \lambda^2)$. Thus the restriction of $X$ to $C$ has characteristic polynomial $(t^2 + \nu^2)(t^2 + \mu^2)$, with $\nu \geq \mu \geq 0$. That is, we have $C = E_{\nu} \oplus E_{\mu}$. Let $e_4$ be a unit vector in $E_{\nu} \subseteq C$. That is, there exists $\wt e_4 \in E_{\nu}$ such that $\{ e_4, \wt e_4 \}$ is an orthonormal basis of $E_{\nu}$ with
\begin{equation} \label{eq:X-e4}
X (e_4) = \nu \wt e_4, \qquad X(\wt e_4) = - \nu e_4.
\end{equation}
If we define
\begin{equation} \label{eq:e567}
e_5 = e_1 \times e_4, \qquad e_6 = e_2 \times e_4, \qquad e_7 = e_3 \times e_4 = (e_1 \times e_2) \times e_4,
\end{equation}
then $\cB = \{ e_1, \ldots, e_7 \}$ is a $\G$-adapted frame for $\R^7$ in the sense of~\eqref{eq:G2-adapted}. Since $\{ e_1, e_2, e_3 = e_1 \times e_2 \}$ is an oriented orthonormal frame for an associative $3$-plane, we have $e_i \times e_j = e_k$ when $i, j, k$ is a cyclic permutation of $1, 2, 3$. Letting $y = e_4$ and $u, v$ be distinct elements of $\{ e_1, e_2, e_3 \}$ in Proposition~\ref{prop:X-in-g2-identity}, we obtain
$$ X( e_i \times e_4, e_j \times e_4 ) = X(e_i, e_j) + X(e_4, (e_i \times e_j) \times e_4) = X(e_i, e_j) + X(e_4, e_k \times e_4). $$
The above yields the three independent identities
\begin{equation} \label{eq:g2-canonical-form-identities}
X_{56} = -\lambda + X_{47}, \qquad X_{67} = X_{45}, \qquad X_{75} = X_{46}.
\end{equation}

If $\nu = 0$, then $\mu = 0$ as well, and from $C = E_{\nu} \oplus E_{\mu}$, we deduce that $X_{ij} = 0$ for $4 \leq i, j \leq 7$. But then~\eqref{eq:g2-canonical-form-identities} forces $\lambda = 0$, contradicting our assumption that $\lambda > 0$. (That is, we have shown that if $\lambda > 0$, then $\nu > 0$ as well. Hence, $X$ cannot have rank 2, yielding another proof of Corollary~\ref{cor:no-rank-2}.) 

Next, we show that necessarily $X_{45} = X_{46} = 0$. We present two arguments for this. The first, which we give here, is by contradiction, and has a geometric flavour. The second, which is purely algebraic, is given in Remark~\ref{rmk:canonical-form-alternate} below.

Recall from the discussion at the end of the second bullet point of Section~\ref{sec:so7-canonical-form} that if we replace
\begin{align*}
e_1 & \quad \mapsto \quad e_1' = (\cos \alpha) e_1 - (\sin \alpha) e_2, \\
e_2 & \quad \mapsto \quad e_2' = (\sin \alpha) e_1 + (\cos \alpha) e_2,
\end{align*}
which leaves $e_3 = e_1 \times e_2$ invariant, then $[X]_{\cB}$ still has the form~\eqref{eq:g2-canonical-form-temp}. We can exploit this freedom to ensure that $X_{45} = 0$, as follows. Under this rotation of the oriented orthonormal frame $\{ e_1, e_2 \}$, we have
\begin{align*}
e_5 = e_1 \times e_4 & \quad \mapsto \quad e_5' =( (\cos \alpha) e_1 - (\sin \alpha) e_2 ) \times e_4 = (\cos \alpha) e_5 - (\sin \alpha) e_6, \\
e_6 = e_2 \times e_4 & \quad \mapsto \quad e_6' = ( (\sin \alpha) e_1 + (\cos \alpha) e_2 ) \times e_4 =  (\sin \alpha) e_5 + (\cos \alpha) e_6.
\end{align*}
Thus we have
\begin{equation} \label{eq:X45-rotation}
\begin{aligned}
X_{45} = X(e_4, e_5) & \quad \mapsto \quad X_{45}' = X(e_4, e_5') = (\cos \alpha) X_{45} - (\sin \alpha) X_{46}, \\
X_{46} = X(e_4, e_6) & \quad \mapsto \quad X_{46}' = X(e_4, e_6') = (\sin \alpha) X_{45} + (\cos \alpha) X_{46}.
\end{aligned}
\end{equation}
If we initially had $X_{45} = 0$, then we need not rotate the frame. If $X_{45} \neq 0$ for our initial choice, then taking $\cot \alpha = \frac{X_{46}}{X_{45}}$ ensures that in the new frame we have $X'_{45} = 0$. To summarize, we can without loss of generality assume that $X_{45} = 0$, but we have now fixed the choice of $\{ e_1, e_2 \}$ and no longer have the freedom to rotate this frame.

Using $X_{45} = 0$, we compute
$$ 0 = X_{45} = \langle X(e_4), e_5 \rangle = \nu \langle \wt e_4, e_5 \rangle. $$
Since $\nu > 0$, we deduce that $e_5$ is orthogonal to $\wt e_4$. Since $e_5 = e_1 \times e_4$, it is also orthogonal to $e_4$. Thus $e_5 \in C$ lies in the orthogonal complement of $E_{\nu}$, which is $E_{\mu}$. Hence, there exists $\wt e_5 \in E_{\mu}$ such that $\{ e_5, \wt e_5 \}$ is an orthonormal basis of $E_{\mu}$ with
\begin{equation} \label{eq:X-e5}
X (e_5) = \mu \wt e_5, \qquad X(\wt e_5) = - \mu e_5.
\end{equation}
Because both $\{ e_4, e_5, e_6, e_7 \}$ and $\{ e_4, \wt e_4, e_5, \wt e_5 \}$ are orthonormal bases for $C$, it must be the case that there exists an angle $\beta$ such that
\begin{equation} \label{eq:canonical-form-e6e7}
e_6 = (\cos \beta) \wt e_4 \pm (\sin \beta) \wt e_5, \qquad e_7 = - (\sin \beta) \wt e_4 \pm (\cos \beta) \wt e_5,
\end{equation}
where the $\pm$ sign is determined from $e_6 \w e_7 = \pm \wt e_4 \w \wt e_5$.

Using~\eqref{eq:X-e4},~\eqref{eq:X-e5}, and~\eqref{eq:canonical-form-e6e7}, we can compute
\begin{align*}
X_{67} & = \langle X(e_6), e_7 \rangle \\
& = \langle (\cos \beta) X(\wt e_4) \pm (\sin \beta) X(\wt e_5), e_7 \rangle \\
& = \langle - \nu (\cos \beta) e_4 \mp \mu (\sin \beta) e_5, e_7 \rangle = 0,
\end{align*}
confirming the relation $X_{67} = X_{45}$ of~\eqref{eq:g2-canonical-form-identities}, because $X_{45} = 0$. Similarly we have
\begin{align*}
X_{75} & = - X_{57} = - \langle X(e_5), e_7 \rangle = - \mu \langle \wt e_5, e_7 \rangle = - (\pm \mu \cos \beta), \\
X_{46} & = \langle X(e_4), e_6 \rangle = \nu \langle \wt e_4, e_6 \rangle = \nu \cos \beta.
\end{align*}
Thus, the relation $X_{75} = X_{46}$ of~\eqref{eq:g2-canonical-form-identities} becomes
\begin{equation} \label{eq:canonical-form-rel-1}
(\nu \pm \mu) \cos \beta = 0.
\end{equation}
Again, in the same way we compute
\begin{align*}
X_{56} & = \langle X(e_5), e_6 \rangle = \mu \langle \wt e_5, e_6 \rangle = \pm \mu \sin \beta, \\
X_{47} & = \langle X(e_4), e_7 \rangle = \nu \langle \wt e_4, e_7 \rangle = - \nu \sin \beta.
\end{align*}
Thus, the relation $X_{56} = - \lambda + X_{47}$ of~\eqref{eq:g2-canonical-form-identities} becomes
\begin{equation} \label{eq:canonical-form-rel-2}
(\nu \pm \mu) \sin \beta = - \lambda.
\end{equation}
If $\nu \pm \mu = 0$, then~\eqref{eq:canonical-form-rel-2} gives $\lambda = 0$, which is a contradiction. Thus $\nu \pm \mu \neq 0$, so~\eqref{eq:canonical-form-rel-1} gives $\cos \beta = 0$. Thus we must have $|\sin \beta| = 1$. Since $\nu \geq \mu \geq 0$ and $\nu \pm \mu \neq 0$, we always have $\nu \pm \mu > 0$. Moreover, we have $- \lambda < 0$. Hence $|\sin \beta| = 1$ and~\eqref{eq:canonical-form-rel-2} are only consistent if $\sin \beta = - 1$. Substituting $\cos \beta = 0$ and $\sin \beta = - 1$ into the above expressions for the matrix elements, we have established that
\begin{equation*}
X_{45} = X_{67} = 0, \qquad X_{46} = X_{75} = 0, \qquad X_{56} = \mp \mu, \qquad X_{47} = \nu.
\end{equation*}
(As an aside, we note that since $X_{45} = X_{46} = 0$, equation~\eqref{eq:X45-rotation} shows that these entries were necessarily both zero even before we rotated the $\{ e_1, e_2 \}$ frame, so no rotation was necessary. See Remark~\ref{rmk:canonical-form-alternate} for an algebraic proof of this fact.)

Finally, the relation $X_{56} = - \lambda + X_{47}$ of~\eqref{eq:g2-canonical-form-identities} becomes $ \mp \mu = - \lambda + \nu$, which we rewrite as
$$ \lambda = \nu \pm \mu. $$
Since $\lambda \geq \nu \geq \mu \geq 0$, the above is only consistent if we take the $\pm = +$, unless $\mu = 0$ in which case the sign is irrelevant. Thus, in both cases we have $\lambda = \nu + \mu$. To summarize, we have shown that
$$ X_{45} = X_{46} = X_{67} = X_{75} = 0, \qquad X_{56} = -\mu, \qquad X_{47} = \nu, \qquad \text{with $\lambda = \nu + \mu$}. $$
We then have the following three cases:
\begin{enumerate}[(i)]
\item If $\lambda = \mu$, then we necessarily have $\nu = 0$, so $\mu = 0$, and hence $X = 0$.
\item If $\lambda > 0$ and $\mu = 0$, then $\lambda = \nu > 0$ and $\rank X = 4$. Moreover $\ker X$ is spanned by $\{ e_3, e_5, e_6 \}$. Note that by~\eqref{eq:forms-coordinates}, we have $e_5 \times e_6 = -e_3$. Thus $\ker X$ is an associative $3$-plane with oriented orthonormal basis $\{ -e_3, e_5, e_6 \}$.
\item If $\lambda > \mu > 0$, then $\nu > 0$ as well, and $\rank X = 6$. In this case $\ker X$ is spanned by $\{ e_3 \}$.
\end{enumerate}
Since $\nu = \lambda - \mu$, the proof is complete.
\end{proof}

\begin{rmk} \label{rmk:canonical-form-alternate}
In the course of the proof of Theorem~\ref{thm:g2-canonical-form}, we had reduced to the following situation. Let $\nu > 0$, and $X(e_4) = \nu \wt e_4$. Then we ended up showing that $X_{45} = X_{46} = 0$ and $X_{47} = \nu$. Here we give an algebraic demonstration of this fact. For the $4 \times 4$ matrix $[\rest{X}{C}]_{{\cB}_C}$ of~\eqref{eq:g2-canonical-form-temp}, using~\eqref{eq:g2-canonical-form-identities}, we can write
$$ [\rest{X}{C}]_{{\cB}_C} = Y = \begin{pmatrix} 0 & - a & - b & - c \\ a & 0 & -(c+\lambda) & b \\ b & c+\lambda & 0 & -a \\ c & -b & a & 0 \end{pmatrix}. $$
The condition $X(e_4) = \nu \wt e_4$ says that $\nu^2 = a^2 + b^2 + c^2$. Explicitly computing $\det (t I - Y)$ gives the characteristic polynomial of $Y$, which using $\nu^2 = a^2 + b^2 + c^2$ turns out to be
$$ t^4 + (2 \nu^2 + 2 c \lambda + \lambda^2) t^2 + (\nu^2 + c \lambda)^2, $$
with roots
\begin{align*}
t^2 & = - \frac{1}{2} (2 \nu^2 + 2 c \lambda + \lambda^2) \pm \frac{1}{2} \sqrt{ (2 \nu^2 + 2 c \lambda + \lambda^2)^2 - 4 (\nu^2 + c \lambda)^2} \\
& = - \nu^2 - c \lambda - \frac{1}{2} \lambda^2 \pm \frac{1}{2} \sqrt{ (4 \nu^2 + 4 c \lambda) \lambda^2 + \lambda^4}.
\end{align*}
However, we know that the characteristic polynomial of $Y$ is $(t^2 + \nu^2)(t^2 + \mu^2)$, and since $\nu^2 \geq \mu^2$, the smaller root (corresponding to the minus sign above) must equal $- \nu^2$. Thus we have
$$ - \nu^2 = - \nu^2 - c \lambda - \frac{1}{2} \lambda^2 - \frac{1}{2} \sqrt{ (4 \nu^2 + 4 c \lambda) \lambda^2 + \lambda^4}, $$
which can be rearranged to give
\begin{equation} \label{eq:c-f-alternate}
-(2 c \lambda + \lambda^2) = \sqrt{ (4 \nu^2 + 4 c \lambda) \lambda^2 + \lambda^4}.
\end{equation}
Squaring both sides gives
$$ 4 c^2 \lambda^2 + 4 c \lambda^3 + \lambda^4 = (4 \nu^2 + 4 c \lambda) \lambda^2 + \lambda^4, $$
and thus since $\lambda > 0$ we deduce that $c^2 = \nu^2$. This forces $a^2 + b^2 = 0$, so $a=b=0$. Moreover,~\eqref{eq:c-f-alternate} now gives $-(2 c \lambda + \lambda^2) \geq 0$, so $c$ cannot be positive, and thus $c = - \nu$. 
\end{rmk}

\begin{cor} \label{cor:g2-canonical-form}
Let $X \in \Lambda^2_{14} = \lieg$. There exists a $\G$-adapted orthonormal frame $\cB = \{ f_1, \ldots, f_7 \}$ and real numbers $\lambda \geq \nu \geq \mu \geq 0$ with
$$ \lambda = \nu + \mu, $$
such that, with respect to the basis $\cB$, the matrix $[X]_{\cB}$ takes the form
\begin{equation} \label{eq:g2-canonical-form-exp2}
[X]_{\cB} = \begin{pmatrix} 0 & & & \\ & \begin{matrix} 0 & -\mu \\ \mu & 0 \end{matrix} & & \\ & & \begin{matrix} 0 & -\nu \\ \nu & 0 \end{matrix} & \\ & & & \begin{matrix} 0 & -\lambda \\ \lambda & 0 \end{matrix} \end{pmatrix}
\end{equation}
where all nondisplayed entries are zero. Moreover, there are exactly three cases that occur:
\begin{enumerate}[(i)]
\item if $\lambda = \mu$, then $\lambda = \mu = \nu = 0$, and $X = 0$;
\item if $\lambda > 0$ and $\mu = 0$, then $\lambda = \nu$, $\rank X = 4$, and $\ker X$ is the associative $3$-plane spanned by $\{ f_1, f_2, f_3 \}$;
\item if $\lambda > \mu > 0$, then $\rank X = 6$, and $\ker X$ is spanned by $\{ f_1 \}$.
\end{enumerate}
\end{cor}
\begin{proof}
From~\eqref{eq:forms-coordinates}, we see that if $\{ e_1, \ldots, e_7 \}$ is a $\G$-adapted frame, then
$$ -e_3 \times e_5 = e_6, \qquad -e_3 \times e_7 = e_4, \qquad e_5 \times e_7 = e_2, \qquad e_6 \times e_7 = - e_1. $$
Thus, if we define
$$ f_1 = - e_3, \quad f_2 = e_5, \quad f_3 = e_6, \quad f_4 = e_7, \quad f_5 = e_4, \quad f_6 = e_2, \quad f_7 = - e_1, $$
then the above relations give
$$ f_1 \times f_2 = f_3, \qquad f_1 \times f_4 = f_5, \qquad f_2 \times f_4 = f_6, \qquad f_3 \times f_4 = f_7, $$
and thus $\{ f_1, \ldots, f_7 \}$ is also a $\G$-adapted frame. Using the canonical form~\eqref{eq:g2-canonical-form-exp1} in terms of $\{ e_1, \ldots, e_7 \}$ and expressing it in terms of the new $\G$-adapted frame $\{ f_1, \ldots, f_7 \}$ defined above, we obtain~\eqref{eq:g2-canonical-form-exp2}. For example, $X(f_6, f_7) = - X(e_2, e_1) = - \lambda$. The three cases are precisely the cases of Theorem~\ref{thm:g2-canonical-form}.
\end{proof}

\begin{rmk} \label{rmk:g2-canonical-form}
Equation~\eqref{eq:g2-canonical-form-exp2} is formally the same as~\eqref{eq:so7-canonical-form}. However, there are two important points to emphasize. Since $\Lambda^2_{14} \subseteq \Lambda^2$, we know that there will exist an orthonormal basis $\cB$ such that $[X]_{\cB}$ takes the form~\eqref{eq:so7-canonical-form}. What is \emph{not obvious}, and what lies at the heart of Theorem~\ref{thm:g2-canonical-form}, is that we can in fact always achieve the form~\eqref{eq:so7-canonical-form} using a \emph{$\G$-adapted frame}. Moreover, we also obtain the fundamental constraint that $\lambda = \nu + \mu$ for elements of $\Lambda^2_{14}$, which rules out the possibility of rank $2$. (Compare with Remark~\ref{rmk:canonical-form-7} for the $\Lambda^2_7$ case.)
\end{rmk}

\begin{rmk} \label{rmk:maximal-torus}
For readers familiar with Lie theory, Theorem~\ref{thm:g2-canonical-form} is nothing more than an explicit demonstration of the maximal torus theorem, at the Lie algebra level, in the case of the exceptional Lie algebra $\lieg$. More precisely, let $\mathfrak{t}$ denote the set of elements in $\lieg$ whose matrix with respect to the standard basis (which is $\G$-adapted) take the form of the right hand side of~\eqref{eq:g2-canonical-form-exp2} with $\lambda = \nu + \mu$. It is easy to check that $\mathfrak{t}$ forms a $2$-dimensional abelian subalgebra of $\lieg$, and in fact $\mathfrak{t}$ is a \emph{maximal} abelian subalgebra. Theorem~\ref{thm:g2-canonical-form} shows that any $X \in \lieg$ is \emph{conjugate} by an element of $\G$ to an element of $\mathfrak{t}$. In the language of Lie theory, this says that $\mathfrak{t}$ is a \emph{maximal torus} for $\lieg$, and the fact that $\dim \mathfrak{t} = 2$ says that the \emph{rank} of $\lieg$ is $2$. The advantage of our proof of Theorem~\ref{thm:g2-canonical-form} is that it uses no abstract Lie theory, but rather explicitly uses the algebra of the cross product on $\R^7$.
\end{rmk}

\begin{rmk} \label{rmk:anton}
In~\cite{Iliashenko}, Anton Iliashenko applies Theorem~\ref{thm:g2-canonical-form} to improve the estimates of Bourguignon--Karcher~\cite{BK} for the eigenvalues of the self-adjoint Riemann curvature operator $\hat{R} \colon \Omega^2 \to \Omega^2$ in the case of \emph{nearly $\G$ manifolds}, when $\hat R$ is restricted to $\Omega^2_{14}$.
\end{rmk}

\section{Associative $3$-planes and $\so(4)$ subalgebras} \label{sec:assoc-so4}

In this section we demonstrate that given an associative $3$-plane $P \subseteq \R^7$, we obtain a Lie subalgebra $\Theta(P)$ of $\so(7) = \Lambda^2 (\R^7)$ that is isomorphic to $\so(4)$. We also completely describe the intersections $\Theta(P) \cap \Theta(Q)$ for two distinct associative $3$-planes $P, Q$.

\subsection{Structure of the splitting $\Lambda^2 = \Lambda^2_7 \oplus \Lambda^2_{14}$} \label{sec:structure-splitting}

Recall the orthogonal decomposition $\Lambda^2 = \Lambda^2_7 \oplus \Lambda^2_{14}$ described in Section~\ref{sec:lambda-2}. We know that $\Lambda^2_{14} = \lieg$, so it is a Lie subalgebra. In fact, it is straightforward to verify directly that $\Lambda^2_{14}$ is closed under the Lie bracket of $\Lambda^2 = \so(7)$, using~\eqref{eq:g2-deriv} and~\eqref{eq:Lambda-2-7-14-frame}. Let us investigate the bracket of an element of $\Lambda^2_{14}$ with an element of $\Lambda^2_7$, or of two elements of $\Lambda^2_7$. In particular, the failure of $\Lambda^2_7$ to be a Lie subalgebra motivates the definition of the object $\Psi_{uv}$ in~\eqref{eq:Psi-defn} below, which plays a crucial role throughout Section~\ref{sec:assoc-so4}.

First, if $X \in \Lambda^2_{14}$ and $u \hk \ph \in \Lambda^2_7$, then using~\eqref{eq:g2-deriv} we have
\begin{align*}
[X, u \hk \ph]_{ij} & = X_{ip} (u \hk \ph)_{pj} - (u \hk \ph)_{ip} X_{pj} \\
& = X_{ip} u_k \ph_{kpj} - u_k \ph_{kip} X_{pj} \\
& = u_k ( X_{ip} \ph_{pjk} + X_{jp} \ph_{ipk} ) \\
& = u_k ( - X_{kp} \ph_{ijp} ) \\
& = (X_{pk} u_k) \ph_{pij}.  
\end{align*}
Thus we have shown that
$$ \text{$[X, u \hk \ph] = X(u) \hk \ph$ for $X \in \Lambda^2_{14}$. In particular, $[\Lambda^2_{14}, \Lambda^2_7] \subseteq \Lambda^2_7$.} $$
Note that an invariant (frame-independent) proof of the above also follows immediately from~\eqref{eq:X-in-g2-alt} and the fact that $(u \hk \ph) w = u \times w$.

It remains to consider the Lie bracket of two elements of $\Lambda^2_7$. We compute
\begin{align*}
[ u \hk \ph, v \hk \ph ]_{ij} & = (u \hk \ph)_{ik} (v \hk \ph)_{kj} - (v \hk \ph)_{ik} (u \hk \ph)_{kj} \\
& = (u_p \ph_{pik}) (v_q \ph_{qkj}) - (v_q \ph_{qik}) (u_p \ph_{pkj}) \\
& = u_p v_q \ph_{pik} \ph_{jqk} - u_p v_q \ph_{pjk} \ph_{iqk}.
\end{align*}
Applying~\eqref{eq:fundamental}, we obtain
\begin{align} \nonumber
[ u \hk \ph, v \hk \ph ]_{ij} & = u_p v_q ( \delta_{pj} \delta_{iq} - \delta_{pq} \delta_{ij} - \ps_{pijq} ) - u_p v_q ( \delta_{pi} \delta_{jq} - \delta_{pq} \delta_{ji} - \ps_{pjiq} ) \\ \label{eq:77-temp}
& = u_j v_i - u_i v_j - 2 u_p v_q \ps_{pqij}.
\end{align}
In order to understand equation~\eqref{eq:77-temp}, we are motivated to make the following definition.

Let $u, v \in \R^7$. Define $\Psi_{uv} \in \Lambda^2$ by
\begin{equation} \label{eq:Psi-defn}
\Psi_{uv} := v \hk u \hk \ps = \sta (u \w v \w \ph).
\end{equation}
In terms of a local orthonormal frame, we have
\begin{equation} \label{eq:Psi-uv-coords}
(\Psi_{uv})_{ij} = u_p v_q \ps_{pqij}.
\end{equation}
As a $2$-form, $X = u \w v$ has components $X_{pq} = u_p v_q - u_q v_p$, so
$$ \ps_{ijpq} X_{pq} = (u_p v_q - u_q v_p) \ps_{pqij} = 2 u_p v_q \ps_{pqij}. $$
Thus, comparing the above two equations with~\eqref{eq:Pbeta}, we deduce that
\begin{equation} \label{eq:Psi-uv}
\Psi_{uv} = - 2 (u \w v)_7 + (u \w v)_{14}.
\end{equation}
From~\eqref{eq:Psi-uv} and $u \w v = (u \w v)_7 + (u \w v)_{14}$ we obtain
\begin{equation} \label{eq:uv-components}
( u \w v )_7 = \tfrac{1}{3} ( u \w v - \Psi_{uv}), \qquad ( u \w v )_{14} = \tfrac{1}{3} ( 2 u \w v + \Psi_{uv} ).
\end{equation}

Returning now to equation~\eqref{eq:77-temp}, and using~\eqref{eq:Psi-uv-coords} and~\eqref{eq:Psi-uv}, we find that
\begin{align*}
[ u \hk \ph, v \hk \ph ] & = - (u \w v) - 2 \Psi_{uv} \\
& = - (u \w v)_7 - (u \w v)_{14} -2 ( -2 (u \w v)_7 + (u \w v)_{14} ),
\end{align*}
and hence we conclude that
\begin{equation} \label{eq:77-temp2}
[ u \hk \ph, v \hk \ph ] = 3 (u \w v)_7 - 3 (u \w v)_{14}.
\end{equation}
Equation~\eqref{eq:77-temp2} shows that $[ u \hk \ph, v \hk \ph ]$ in general has components in both $\Lambda^2_7$ \emph{and} $\Lambda^2_{14}$. In fact, unless $u, v$ are linearly dependent where the bracket is zero, the bracket $[ u \hk \ph, v \hk \ph ]$ always has \emph{nonzero} components in both $\Lambda^2_7$ and $\Lambda^2_{14}$. To see this, suppose that $u \w v$ is either purely type $\Lambda^2_7$ or $\Lambda^2_{14}$. Then we have
$$ u \w v \w \ph = c \sta (u \w v) $$
for $c = -2, 1$, respectively. Wedging both sides with $u \w v$ gives $0 = c |u \w v|^2$, so $u \w v = 0$, and $u, v$ are linearly dependent.

From~\eqref{eq:uv-components} and~\eqref{eq:77-temp2} we obtain
\begin{equation} \label{eq:77-summary}
\begin{aligned}
[ u \hk \ph, v \hk \ph ] & = - u \w v - 2 \Psi_{uv}, \\
\pi_7 ( [ u \hk \ph, v \hk \ph ] ) & = u \w v - \Psi_{uv}, \\
\pi_{14} ( [ u \hk \ph, v \hk \ph ] ) & = - 2 u \w v - \Psi_{uv},
\end{aligned}
\end{equation}
giving the components of the bracket $[u \hk \ph, v \hk \ph]$ explicitly in terms of $u \w v$ and $\Psi_{uv}$.

Another relation which we use in Section~\ref{sec:Psi-identities} is the following, which follows immediately from the first equation in~\eqref{eq:uv-components} and equation~\eqref{eq:pi7-vw}, namely
\begin{equation} \label{eq:Psi-identity}
u \w v = (u \times v) \hk \ph + \Psi_{uv}.
\end{equation}
Finally, from~\eqref{eq:77-temp2} and~\eqref{eq:pi7-vw} we obtain the relation
\begin{equation} \label{eq:pi7-of-77}
\pi_7 ( [ u \hk \ph, v \hk \ph ] ) = ( u \times v ) \hk \ph.
\end{equation}

\subsection{Identities satisfied by $\Psi_{uv}$} \label{sec:Psi-identities}

Recall the definition~\eqref{eq:Psi-defn} of $\Psi_{uv} \in \Lambda^2$ for any $u, v \in \R^7$. We establish several important identities satisfied by $\Psi_{uv}$.

\begin{lemma} \label{lemma:Psi-inner-product}
Let $u, v, w, y \in \R^7$. Then we have
\begin{equation} \label{eq:Psi-inner-product}
\langle \Psi_{uv}, \Psi_{wy} \rangle = 4 \langle u \w v, w \w y \rangle - 2 \ps(u, v, w, y).
\end{equation}
In particular, if $\{ u, v, w \}$ is orthonormal in $\R^7$, then $\{ \frac{1}{2} \Psi_{vw}, \frac{1}{2} \Psi_{wu}, \frac{1}{2} \Psi_{uv} \}$ is orthonormal in $\Lambda^2$. We also have
\begin{equation} \label{eq:Psi-inner-product-b}
\langle \Psi_{uv}, w \w y \rangle = 2 \ps(u, v, w, y).
\end{equation}
\end{lemma}
\begin{proof}
Using~\eqref{eq:Psi-uv-coords} and~\eqref{eq:psps}, we compute
\begin{align*}
\langle \Psi_{uv}, \Psi_{wy} \rangle & = (u_a v_b \ps_{abij}) (w_p y_q \ps_{pqij}) \\
& = u_a v_b w_p y_q (4 \delta_{ap} \delta_{bq} - 4 \delta_{aq} \delta_{bp} - 2 \ps_{abpq}) \\
& = 4 \langle u, w \rangle \langle v, y \rangle - 4 \langle u, y \rangle \langle v, w \rangle - 2 \ps(u, v, w, y),
\end{align*}
which is precisely~\eqref{eq:Psi-inner-product}. The second statement follows, since the $\ps$ term vanishes in each case. Finally, equation~\eqref{eq:Psi-inner-product-b} is immediate from~\eqref{eq:Psi-uv-coords} and $(w \w y)_{ij} = w_i y_j - w_j y_i$.
\end{proof}

\begin{lemma} \label{lemma:identities}
Let $u, v, w \in \R^7$ be \emph{orthogonal}. The following identities hold:
\begin{align} \label{eq:identity-1}
(u \times w) \times (v \times w) & = 2 \ph(u,v,w) w - |w|^2 u \times v, \\ \label{eq:identity-2}
\Psi_{(u \times w)(v \times w)} & = \ps(u,v,w) \w w + |w|^2 u \w v - \ph(u,v,w) w \hk \ph, \\ \label{eq:identity-3}
(u \times w) \w (v \times w) & = \ph(u,v,w) w \hk \ph + \ps(u,v,w) \w w + |w|^2 \Psi_{uv}.
\end{align}
\end{lemma}
\begin{proof}
Replacing $u$ with $u \times w$ in~\eqref{eq:iterated-cross}, we have
\begin{align} \nonumber
(u \times w) \times (v \times w) & = - \langle u \times w, v \rangle w + \langle u \times w, w \rangle + \ps(u \times w, v, w) \\ \label{eq:identities-temp1}
& = \ph(u,v,w) w + 0 + \ps(u \times w, v, w).
\end{align}
Expanding the last term of~\eqref{eq:iterated-cross} with $u \times w = u_a w_b \ph_{abp} e_p$, we have
$$ \ps(u \times w, v, w) = u_a w_b \ph_{abp} \ps(e_p, v, w) = u_a w_b v_i w_j \ph_{abp} \ps_{pijk} e_k. $$
Using the first identity of~\eqref{eq:phps}, the above becomes
\begin{align*}
\ps(u \times w, v, w) & = - u_a w_b v_i w_j (\ph_{abp} \ps_{ijkp}) e_k \\
& = - u_a w_b v_i w_j ( \delta_{ai} \ph_{bjk} + \delta_{aj} \ph_{ibk} + \delta_{ak} \ph_{ijb} - \delta_{bi} \ph_{ajk} - \delta_{bj} \ph_{iak} - \delta_{bk} \ph_{ija} ) e_k.
\end{align*}
Using the skew-symmetry of $\ph$ and the orthogonality of $u, v, w$, the first four terms vanish, leaving
$$ \ps(u \times w, v, w) = |w|^2 v \times u + \ph(v, w, u) w = - |w|^2 u \times v + \ph(u, v, w) w. $$
Substituting the above into~\eqref{eq:identities-temp1} yields~\eqref{eq:identity-1}.

Similarly, we compute
\begin{align*}
( \Psi_{(u \times w)(v \times w)} )_{ab} & = \ps( u \times w, v \times w, e_a, e_b ) = u_i w_j \ph_{ijk} v_l w_m \ph_{lmn} \ps_{knab} \\
& = u_i w_j v_l w_m \ph_{ijk} (\ph_{lmn} \ps_{kabn}) \\
& = u_i w_j v_l w_m \ph_{ijk} (\delta_{lk} \ph_{mab} + \delta_{la} \ph_{kmb} + \delta_{lb} \ph_{kam} - \delta_{mk} \ph_{lab} - \delta_{ma} \ph_{klb} - \delta_{mb} \ph_{kal}) \\
& = \ph(u, w, v) w_m \ph_{mab} + u_i w_j v_a w_m \ph_{ijk} \ph_{mbk} + u_i w_j v_b w_m \ph_{ijk} \ph_{amk} \\
& \qquad {} - \ph(u, w, w) v_l \ph_{lab} - u_i w_j v_l w_a \ph_{ijk} \ph_{lbk} - u_i w_j v_l w_b \ph_{ijk} \ph_{alk}.
\end{align*}
The fourth term above vanishes. Applying~\eqref{eq:fundamental} four times, we have
\begin{align*}
( \Psi_{(u \times w)(v \times w)} )_{ab} & = - \ph(u, v, w) (w \hk \ph)_{ab} + u_i w_j v_a w_m (\delta_{im} \delta_{jb} - \delta_{ib} \delta_{jm} - \ps_{ijmb}) \\
& \qquad {} + u_i w_j v_b w_m (\delta_{ia} \delta_{jm} - \delta_{im} \delta_{ja} - \ps_{ijam}) - u_i w_j v_l w_a (\delta_{il} \delta_{jb} - \delta_{ib} \delta_{jl} - \ps_{ijlb}) \\
& \qquad {} - u_i w_j v_l w_b (\delta_{ia} \delta_{jl} - \delta_{il} \delta_{ja} - \ps_{ijal}).
\end{align*}
Using the orthogonality of $u, v, w$ and the skew-symmetry of $\ps$, most of the terms above vanish. We are left with
\begin{align*}
( \Psi_{(u \times w)(v \times w)} )_{ab} & = - \ph(u, v, w) (w \hk \ph)_{ab} - |w|^2 u_b v_a + |w|^2 u_a v_b + w_a \ps(u, w, v)_b - w_b \ps(u,w,v)_a \\
& = - \ph(u, v, w) (w \hk \ph)_{ab} + |w|^2 (u \w v)_{ab} + (\ps(u, v, w) \w w)_{ab},
\end{align*}
which is the component form of~\eqref{eq:identity-2}.

Next, using~\eqref{eq:Psi-identity} with $u$ replaced by $u \times w$ and $v$ replaced by $v \times w$, we have
$$ (u \times w) \w (v \times w) = \big( (u \times w) \times (v \times w) \big) \hk \ph + \Psi_{(u \times w)(v \times w)}. $$
Substituting the identities~\eqref{eq:identity-1} and~\eqref{eq:identity-2}, this becomes
\begin{align*}
(u \times w) \w (v \times w) & = ( 2 \ph(u, v, w) w - |w|^2 u \times v) \hk \ph \\
& \qquad {} + \ps(u,v,w) \w w + |w|^2 u \w v - \ph(u,v,w) w \hk \ph \\
& = \ph(u, v, w) w \hk \ph + \ps(u, v, w) \w w + |w|^2 \big( u \w v - (u \times v) \hk \ph \big).
\end{align*}
Equation~\eqref{eq:identity-3} now follows from the above and~\eqref{eq:Psi-identity}.
\end{proof}

The key result that enables us to construct $\Theta(P)$ in Section~\ref{sec:Theta-P} is the following identity.

\begin{prop} \label{prop:main-identity}
Let $u, v, w \in \R^7$ be \emph{orthogonal}. Under the identification between $2$-forms and skew-adjoint endomorphisms, the commutator $[ \Psi_{uw}, \Psi_{vw} ]$ satisfies the remarkable identity
\begin{equation} \label{eq:main-identity}
[ \Psi_{uw}, \Psi_{vw} ] = - |w|^2 \Psi_{uv} - \Psi_{(u \times w)(v \times w)}.
\end{equation}
\end{prop}
\begin{proof}
We have
\begin{align*}
- [ \Psi_{uw}, \Psi_{vw} ]_{kc} & = - (\Psi_{uw})_{kl} (\Psi_{vw})_{lc} + (\Psi_{vw})_{kl} (\Psi_{uw})_{lc} \\
& = - (u_i w_j \ps_{ijkl}) (v_a w_b \ps_{ablc}) + (v_a w_b \ps_{abkl}) (u_i w_j \ps_{ijlc}) \\
& = u_i w_j v_a w_b \ps_{ijkl} \ps_{abcl} - u_i w_j v_a w_b \ps_{ijcl} \ps_{abkl}.
\end{align*}
The second term above (apart from the minus sign) is identical to the first term after interchanging $k \leftrightarrow c$. As shorthand, we write
\begin{equation} \label{eq:main-identity-temp}
 - [ \Psi_{uw}, \Psi_{vw} ]_{kc} = u_i w_j v_a w_b \ps_{ijkl} \ps_{abcl} - (k \leftrightarrow c).
 \end{equation}
Using the identity~\eqref{eq:psps}, the first term on the right hand side above is
\begin{align*}
u_i w_j v_a w_b & \big( - \ph_{ajk} \ph_{ibc} - \ph_{iak} \ph_{jbc} - \ph_{ija} \ph_{kbc} \\
& \qquad {} + \delta_{ia} \delta_{jb} \delta_{kc} + \delta_{ib} \delta_{jc} \delta_{ka} + \delta_{ic} \delta_{ja} \delta_{kb} - \delta_{ia} \delta_{jc} \delta_{kb} - \delta_{ib} \delta_{ja} \delta_{kc} - \delta_{ic} \delta_{jb} \delta_{ka} \\
& \qquad {} - \delta_{ia} \ps_{jkbc} - \delta_{ja} \ps_{kibc} - \delta_{ka} \ps_{ijbc} + \delta_{ab} \ps_{ijkc} - \delta_{ac} \ps_{ijkb} \big).
\end{align*}
By the orthogonality of $u, v, w$ and the skew-symmetry of $\ps$ and $\ph$, the second term on the first line vanishes, all the terms except the last one on the second line vanish, and all the terms on the third line vanish. We are left with
\begin{align*}
& \qquad {} u_i w_j v_a w_b (- \ph_{ajk} \ph_{ibc} - \ph_{ija} \ph_{kbc} - \delta_{ic} \delta_{jb} \delta_{ka}) \\
& =-  (v \times w)_k (u \times w)_c + \ph(u, w, v) (w \hk \ph)_{kc} - |w|^2 u_c v_k.
\end{align*}
Substituting the above into~\eqref{eq:main-identity-temp}, we find that
\begin{align*}
 - [ \Psi_{uw}, \Psi_{vw} ]_{kc} & = - (v \times w)_k (u \times w)_c + (v \times w)_c (u \times w)_k \\
 & \qquad {} + 2 \ph(u, w, v) (w \hk \ph)_{kc} + |w|^2 (u_k v_c - u_c v_k),
\end{align*}
which is the component form of
\begin{equation} \label{eq:main-identity-temp2}
 - [ \Psi_{uw}, \Psi_{vw} ] = (u \times w) \w (v \times w) - 2 \ph(u, v, w) w \hk \ph + |w|^2 u \w v.
\end{equation}
Substituting~\eqref{eq:identity-3} into the above, we obtain
$$ - [ \Psi_{uw}, \Psi_{vw} ] = \ps(u, v, w) \w w + |w|^2 \Psi_{uv} - \ph(u, v, w) w \hk \ph + |w|^2 u \w v. $$
Subtracting~\eqref{eq:identity-2} from the above, we are left with
$$ - [ \Psi_{uw}, \Psi_{vw} ] - \Psi_{(u \times w)(v \times w)} = |w|^2 \Psi_{uv}, $$
which is precisely~\eqref{eq:main-identity}.
\end{proof}

\subsection{The $\so(3)$ subalgebra $\Psi(P)$ of an associative $3$-plane $P$} \label{sec:Psi-P}

Let $P$ be a $3$-plane in $\R^7$, that is a $3$-dimensional linear subspace. Then $\Lambda^2 (P)$ is a $3$-dimensional linear subspace of $\so(7) = \Lambda^2 (\R^7)$. Let $\{ u, v, w \}$ be an orthonormal basis for $P$. Then $\{ v \w w, w \w u, u \w v \}$ is an orthonormal basis for $\Lambda^2 (P)$. Moreover, using~\eqref{eq:2form-map}, for any $y \in P$ we have
\begin{align*}
[ u \w v, w \w v] y & = (u \w v) (w \w v) (y) - ( u \leftrightarrow w) \\
& = (u \w v) ( \langle w, y \rangle v - \langle v, y \rangle w ) - ( u \leftrightarrow w) \\ 
& = \langle w, y \rangle (\langle u, v \rangle v - \langle v, v \rangle u) - \langle v, y \rangle (\langle u, w \rangle v - \langle v, w \rangle u) - ( u \leftrightarrow w) \\
& = - \langle w, y \rangle u + \langle u, y \rangle w \\
& = (u \w w) y,
\end{align*}
and similarly for cyclic permutations of $u, v, w$, yielding
$$ [ u \w v, v \w w ] = w \w u, \qquad [ v \w w, w \w u ] = u \w v, \qquad [ w \w u, u \w v ] = w \w v. $$
If we define $X_1 = v \w w$, $X_2 = w \w u$, and $X_3 = u \w v$, then the above is nothing more than the well-known fact that $\Lambda^2 (P) = \mathrm{Span} \{ X_1, X_2, X_3 \}$ is a Lie subalgebra of $\Lambda^2 (\R^7)$, isomorphic to $\so(3)$, since $[X_i, X_j] = X_k$ for $i, j, k$ a cyclic permutation of $1,2,3$.

Given the $\G$-structure $\ph$ on $\R^7$, we claim that $P$ determines \emph{another} $\so(3)$ subalgebra $\Psi(P)$ of $\so(7)$, which is orthogonal to $\Lambda^2 (P)$. This is obtained as follows. Define a linear map $\Psi \colon \Lambda^2 \to \Lambda^2$ by
$$ \Psi(\beta) = \sta (\beta \w \ph). $$
In fact,~\eqref{eq:Lambda-2-7-14} says that $\Psi$ is a linear isomorphism with $\Psi = -2 \pi_7 + \pi_{14}$. Note also from~\eqref{eq:Psi-defn} that $\Psi(u \w v) = \Psi_{uv}$. Since $\Psi$ is linear, it maps the $3$-dimensional subspace $\Lambda^2 (P)$ of $\Lambda^2 = \Lambda^2(\R^7)$ onto a $3$-dimensional subspace $\Psi (\Lambda^2 (P))$, which we denote by $\Psi (P)$ for simplicity. In fact, since any $2$-form on $P$ is decomposable, \emph{every} element of $\Psi(P)$ is of the form $\Psi_{uv}$ for some (non-unique) $u, v \in P$. Thus,
\begin{equation} \label{eq:Psi-P-defn}
\Psi(P) = \{ \Psi_{uv} \mid u, v \in P \}.
\end{equation}
By~\eqref{eq:Psi-P-defn} and Lemma~\ref{lemma:Psi-inner-product}, the set $\{ \frac{1}{2} \Psi_{vw}, \frac{1}{2} \Psi_{wu}, \frac{1}{2} \Psi_{uv} \}$ is an orthonormal basis of $\Psi(P)$.

\begin{lemma} \label{lemma:Psi-P-orthogonal}
Let $P$ be a $3$-plane in $\R^7$. The two $3$-dimensional subspaces $\Lambda^2 (P)$ and $\Psi(P)$ of $\Lambda^2 (\R^7)$ are orthogonal to each other.
\end{lemma}
\begin{proof}
Given any element of $\{ v \w w, w \w u, u \w v \}$ and any element of $\{ \frac{1}{2} \Psi_{vw}, \frac{1}{2} \Psi_{wu}, \frac{1}{2} \Psi_{uv} \}$, there will be at least one of $u, v, w$ that appears in both, so by~\eqref{eq:Psi-inner-product-b} and the skew-symmetry of $\ps$, these elements will be orthogonal. The result now follows by bilinearity.
\end{proof}

For a general $3$-plane $P$, we cannot say more than this. From now on, suppose that $P$ is \emph{associative}. In this case we can choose $\{ u, v, w \}$ so that
\begin{equation} \label{eq:P-assoc-basis}
\text{$v \times w = u$, $w \times u = v$, and $u \times v = w$.}
\end{equation}

\begin{prop} \label{prop:Psi-P-so3}
Let $P$ be an \emph{associative} $3$-plane in $\R^7$. Then the $3$-dimensional subspace $\Psi(P)$ of $\Lambda^2 (\R^7)$ is a Lie subalgebra, isomorphic to $\so(3)$.
\end{prop}
\begin{proof}
Using the key identity~\eqref{eq:main-identity} and~\eqref{eq:P-assoc-basis}, we compute
\begin{align*}
\big[ \tfrac{1}{2} \Psi_{vw}, \tfrac{1}{2} \Psi_{wu} \big] & = - \tfrac{1}{4} [ \Psi_{vw}, \Psi_{uw} ] = \tfrac{1}{4} [ \Psi_{uw}, \Psi_{vw} ] \\
& = \tfrac{1}{4} ( - |w|^2 \Psi_{uv} - \Psi_{(u \times w)(v \times w)} ) \\
& = - \tfrac{1}{4} (\Psi_{uv} + \Psi_{(-v)u}) \\
& = - \tfrac{1}{2} \Psi_{uv}.
\end{align*}
Cyclically permuting $u,v,w$, we deduce that
$$ \big[ \tfrac{1}{2} \Psi_{vw}, \tfrac{1}{2} \Psi_{wu} \big] = - \tfrac{1}{2} \Psi_{uv}, \qquad \big[ \tfrac{1}{2} \Psi_{wu}, \tfrac{1}{2} \Psi_{uv} \big] = - \tfrac{1}{2} \Psi_{vw}, \qquad \big[ \tfrac{1}{2} \Psi_{uv}, \tfrac{1}{2} \Psi_{vw} \big] = - \tfrac{1}{2} \Psi_{wu}. $$
If we define $Y_1 = \Psi_{vw}$, $Y_2 = \Psi_{wu}$, and $Y_3 = - \Psi_{uv}$, then the above says $[Y_i, Y_j] = Y_k$ for $i, j, k$ a cyclic permutation of $1,2,3$. Thus indeed $\Psi(P) = \mathrm{Span} \{ Y_1, Y_2, Y_3 \}$ is a Lie subalgebra of $\Lambda^2 (\R^7)$, isomorphic to $\so(3)$.
\end{proof}

\begin{rmk} \label{rmk:Psi-P-orientation}
If $P$ is given an orientation, so that $\{ u, v, w \}$ is an oriented orthonormal basis of $P$, then it naturally induces an orientation on $\Lambda^2 (P)$ such that $\{ v \w w, w \w u, u \w v \}$ is an oriented orthonormal basis of $\Lambda^2 (P)$. Then the linear isomorphism $\Psi \colon \Lambda^2 \to \Lambda^2$ maps $\Lambda^2 (P)$ onto $\Psi(P)$, inducing an orientation on $\Psi (P)$ such that $\{ \frac{1}{2} \Psi_{vw}, \frac{1}{2} \Psi_{wu}, \frac{1}{2} \Psi_{uv} \}$ is an oriented orthonormal basis of $\Psi(P)$. Because we had to take $Y_3 = - \frac{1}{2} \Psi_{uv}$ in the proof of Proposition~\ref{prop:Psi-P-so3}, the orientation on $\Psi(P)$ which gives the standard $\so(3)$ commutation relations is the opposite of the orientation induced on $\Psi(P)$ by the orientation on $P$.
\end{rmk}

Next we investigate the bracket between elements of the two $\so(3)$ subalgebras $\Lambda^2 (P)$ and $\Psi(P)$.

\begin{prop} \label{prop:bracket-Psi-P}
Let $P$ be an \emph{associative} $3$-plane in $\R^7$. Then the two $\so(3)$ subalgebras $\Lambda^2 (P)$ and $\Psi (P)$ of $\Lambda^2 (\R^7)$ commute. That is, if $X \in \Lambda^2 (P)$ and $Y \in \Psi (P)$, then $[X, Y] = 0$.
\end{prop}
\begin{proof}
If $X \in \Lambda^2 (P)$ and $Y \in \Psi (P)$, then there exist $u, v, w, y \in P$ such that $X = u \w v$ and $Y = \Psi_{wy}$. Using~\eqref{eq:Psi-uv-coords}, we compute
\begin{align*}
[ X, Y ]_{ij} & = X_{ik} Y_{kj} - Y_{ik} X_{kj} \\
& = (u_i v_k - u_k v_i) (w_p y_q \ps_{pqkj} - (w_p y_q \ps_{pqik}) (u_k v_j - u_j v_k) \\
& = u_i \ps(w, y, v)_j - v_i \ps(w, y, u)_j + \ps(w, y, u)_i v_j - \ps(w, y, v)_i u_j \\
& = \big( u \w \ps(w, y, v) - v \w \ps(w, y, u) \big){}_{ij}.
\end{align*}
Since $P$ is associative, the above expression vanishes by~\eqref{eq:assoc-character}.
\end{proof}

\subsection{The $\so(4)$ subalgebra $\Theta(P)$ of an associative $3$-plane $P$} \label{sec:Theta-P}

Given an associative $3$-plane $P$ in $\R^7$, we define
\begin{equation} \label{eq:Theta-P-defn}
\Theta (P) = \Lambda^2 (P) \oplus \Psi (P).
\end{equation}
Note that $\Theta (P)$ is a $6$-dimensional linear subspace of $\Lambda^2 (\R^7)$, and a Lie subalgebra.

\begin{prop} \label{prop:so4-subalgebra}
Let $P \subseteq \R^7$ be a \emph{associative}. The orthogonal direct sum $\Theta (P) = \Lambda^2 (P) \oplus \Psi (P)$ is a Lie subalgebra of $\Lambda^2 (\R^7)$ isomorphic to $\so(4)$.
\end{prop}
\begin{proof}
This follows from Proposition~\ref{prop:bracket-Psi-P} and the well-known isomorphism $\so(4) \cong \so(3) \oplus \so(3)$, where the right hand side is a direct sum of Lie algebras, meaning that they commute.
\end{proof}

We have shown that any associative $3$-plane $P$ in $\R^7$ gives rise to an $\so(4)$ subalgebra $\Theta(P)$ of $\Lambda^2$. It is natural to investigate the relations between $\Theta(P)$ and $\Theta(Q)$ if $P, Q$ are two distinct associative $3$-planes. We first establish some preliminary results, which conclude with an alternative characterization of $\Theta(P)$ in Corollary~\ref{cor:Theta-P-alt} below.

\begin{lemma} \label{lemma:Theta-P-7}
Let $P$ be an associative $3$-plane. Then $\Theta(P) \cap \Lambda^2_7 = P \hk \ph := \{ u \hk \ph \mid u \in P \}$.
\end{lemma}
\begin{proof}
Let $\{ u, v, w \}$ be an orthonormal basis of $P$ with $w = u \times v$. We saw in Section~\ref{sec:Psi-P} that
$$ \{ v \w w, w \w u, u \w v, \Psi_{vw}, \Psi_{wu}, \Psi_{uv} \} $$
is a basis of $\Theta(P) = \Lambda^2(P) \oplus \Psi(P)$. It then follows immediately from~\eqref{eq:77-summary} that
\begin{align*}
& \{ \pi_7( [v \hk \ph, w \hk \ph]),  \pi_7( [w \hk \ph, u \hk \ph]), \pi_7( [u \hk \ph, v \hk \ph]) \} \\
& \quad \cup \, \{ \pi_{14} ( [v \hk \ph, w \hk \ph]),  \pi_{14} ( [w \hk \ph, u \hk \ph]),  \pi_{14} ( [u \hk \ph, v \hk \ph]) \}
\end{align*}
is a basis for $\Theta(P)$. Thus, using~\eqref{eq:pi7-of-77}, we have
\begin{align*}
\Theta(P) \cap \Lambda^2_7 & = \mathrm{Span} \{ \pi_7( [v \hk \ph, w \hk \ph]),  \pi_7( [w \hk \ph, u \hk \ph]),  \pi_7( [u \hk \ph, v \hk \ph]) \} \\
& = \mathrm{Span} \{ (v \times w) \hk \ph, (w \times u) \hk \ph, (u \times v) \hk \ph \} \\
& = \mathrm{Span} \{ u \hk \ph, v \hk \ph, w \hk \ph \},
\end{align*}
which equals $P \hk \ph$.
\end{proof}

\begin{lemma} \label{lemma:proper-V-intersec}
Let $V \subseteq \R^7$ be a subspace. Let $V \hk \ph = \{ v \hk \ph \mid v \in V \}$, which is a subspace of $\Lambda^2 (\R^7)$ of the same dimension as $V$, since $v \mapsto v \hk \ph$ is injective. We also have the subspace $\Lambda^2 (V)$ of $\Lambda^2 (\R^7)$. Thus $\Lambda^2 (V) \cap (V \hk \ph)$ is a subspace of $\Lambda^2 (\R^7)$. There are two cases:
\begin{itemize}
\item If $V = \R^7$, then $\Lambda^2 (V) \cap (V \hk \ph) = \Lambda^2_7$.
\item If $V$ is a \emph{proper} subspace, then $\Lambda^2 (V) \cap (V \hk \ph) = \{ 0 \}$.
\end{itemize}
\end{lemma}
\begin{proof}
The first statement is just~\eqref{eq:27-first}. Suppose $V$ is proper, so $V^{\perp} \neq \{ 0 \}$. Let $w \in V^{\perp}$ be \emph{nonzero}. Suppose $X \in \Lambda^2 (V) \cap (V \hk \ph)$. Then $X = v \hk \ph$ for some $v \in V$. Since $\langle v, w \rangle = 0$, by~\eqref{eq:cross-frame} and~\eqref{eq:iterated-cross-same} we have
$$ w \times (w \hk X) = w \times (w \hk v \hk \ph) = w \times (v \times w) = |w|^2 v. $$
But if $X \in \Lambda^2 (V)$, then for $w \in V^{\perp}$ we have $w \hk X = 0$, so since $w \neq 0$, the above gives $v = 0$. Therefore $X = 0$, and hence indeed $\Lambda^2 (V) \cap (V \hk \ph) = \{ 0 \}$.
\end{proof}

The above lemmas give us an alternative characterization of the $\so(4)$ subalgebra $\Theta(P)$ as follows.

\begin{cor} \label{cor:Theta-P-alt}
Let $P$ be an associative $3$-plane. Then $\Theta(P) = \Lambda^2 (P) \oplus (P \hk \ph)$. Note that this direct sum is \emph{not} orthogonal.
\end{cor}
\begin{proof}
Recall that $\Theta(P) = \Lambda^2 (P) \oplus \Psi(P)$ is $6$-dimensional. Both $\Lambda^2 (P)$ and $(P \hk \ph)$ are $3$-dimensional, and by Lemma~\ref{lemma:proper-V-intersec} they have trivial intersection. Moreover, Lemma~\ref{lemma:Theta-P-7} shows that $P \hk \ph \subseteq \Theta(P)$. Thus the assertion follows.
\end{proof}

Now suppose that $P$ and $Q$ are both associative $3$-planes. If $P \neq Q$, then the intersection $P \cap Q$ can only have dimension $0$ or $1$, because they are both closed under the cross product, so if they have two orthonormal vectors in common, then they must have three.

\begin{prop} \label{prop:PQ-in7}
Let $P, Q$ be \emph{distinct} associative $3$-planes. Then $\Theta(P) \cap \Theta(Q) \subseteq \Lambda^2_7$.
\end{prop}
\begin{proof}
We prove the contrapositive. Let $X \in \Theta(P) \cap \Theta(Q)$. By Corollary~\ref{cor:Theta-P-alt}, there exists $u_1, u_2, u_3 \in P$ and $v_1, v_2, v_3 \in Q$ such that
\begin{equation} \label{eq:PQ-in7-temp}
X = u_1 \w u_2 + u_3 \hk \ph = v_1 \w v_2 + v_3 \hk \ph.
\end{equation}
Let $V = P + Q$, which is at most $6$-dimensional and thus a proper subspace. By Lemma~\ref{lemma:proper-V-intersec}, we have $X \in \Lambda^2 (V) \oplus (V \hk \ph)$. Thus we deduce from~\eqref{eq:PQ-in7-temp} that $u_1 \w u_2 = v_1 \w v_2$. If $X \notin \Lambda^2_7$, then $u_1 \w u_2 = v_1 \w v_2$ is nonzero, and hence $\mathrm{Span} \{ u_1, u_2 \} \subseteq P$ equals $\mathrm{Span} \{ v_1, v_2 \} \subseteq Q$, so $P = Q$ as explained immediately before the proposition.
\end{proof}

We can now completely describe the intersection $\Theta(P) \cap \Theta(Q)$ when $P \neq Q$.

\begin{cor} \label{cor:PQ-final}
Let $P, Q$ be \emph{distinct} associative $3$-planes. Then $\Theta(P) \cap \Theta(Q) = (P \cap Q) \hk \ph$.
\end{cor}
\begin{proof}
Using Proposition~\ref{prop:PQ-in7} and Lemma~\ref{lemma:Theta-P-7}, we have
\begin{align*}
\Theta(P) \cap \Theta(Q) & = \Theta(P) \cap \Theta(Q) \cap \Lambda^2_7 \\
& = (\Theta(P) \cap \Lambda^2_7) \cap (\Theta(Q) \cap \Lambda^2_7) \\
& = (P \hk \ph) \cap (Q \hk \ph).
\end{align*}
But since $v \mapsto v \hk \ph$ is injective, we have $(P \hk \ph) \cap (Q \hk \ph) = (P \cap Q) \hk \ph$.
\end{proof}

Putting together all of the results in this section, we have established the following.

\begin{thm} \label{thm:so4-final}
Let $P$ be an associative $3$-plane. Then $\Theta(P) = \Lambda^2 (P) \oplus \Psi(P)$ is a Lie subalgebra of $\so(7) = \Lambda^2 (\R^7)$ isomorphic to $\so(4)$. Moreover, the direct sum is orthogonal. We can also describe $\Theta(P)$ as the \emph{non-orthogonal} direct sum $\Theta(P) = \Lambda^2 (P) \oplus (P \hk \ph)$. Finally, if $P, Q$ are \emph{distinct} associative $3$-planes, then the intersection $\Theta(P) \cap \Theta(Q)$ is the Lie subalgebra
\begin{itemize}
\item $\Theta(P) \cap \Theta(Q) = \{ 0 \}$ if $P \cap Q = \{ 0 \}$,
\item $\Theta(P) \cap \Theta(Q) = \mathrm{Span} \{ v \hk \ph \} \cong \so(2)$ if $P \cap Q$ is $1$-dimensional with orthonormal basis $\{ v \}$.
\end{itemize}
\end{thm}

Note that in the second case of Theorem~\ref{thm:so4-final}, $v \hk \ph$ generates the element of $\SO{7}$ which acts as follows. Let $L = \mathrm{Span} \{ v \}$. There are orthogonal decompositions $P =L \oplus \wt P$, $Q = L \oplus \wt Q$, and $\R^7 = L \oplus \wt P \oplus \wt Q \oplus U$. Then $\exp (v \hk \ph)$ fixes $L$ and rotates each of $\wt P$, $\wt Q$, and $U$ as described by exponentiating~\eqref{eq:canonical-form-7}.

\begin{rmk} \label{rmk:different-so4}
There are two other well-known ways to obtain an $\so(4)$ subalgebra given an associative $3$-plane $P$ in $\R^7$. The first is as follows. A  choice of associative $3$-plane $P$ determines an orthogonal splitting $\R^7 = P \oplus P^{\perp}$, and since $P$ is associative we can identify $\R^7 = \imag \Oc$ with $\imag \Qu \oplus \Qu$. Then $\SO{4} = \{ (p,q) \in \Qu \times \Qu \mid |p| = |q| = 1 \} / \{\pm (1,1)\}$ acts on $\R^7 = \imag \Qu \oplus \Qu$ by
$$ (p,q) (v, a) = (p v \overline{p}, p a \overline{q}). $$
(See Harvey--Lawson~\cite{HL} for more details.) One can show that this gives an embedding of $\SO{4}$ in $\G$, and thus an inclusion of Lie algebras $\so(4) \subseteq \lieg$. However, the $\so(4)$ subalgebra $\Theta(P)$ constructed in this section from an associative $3$-plane $P$ is \emph{not} the same, because $\Theta(P)$ does not lie entirely in $\lieg$, but rather lies in the larger Lie algebra $\so(7) = \Lambda^2_7 \oplus \lieg$.

Another possibility is to take the orthogonal complement $P^{\perp}$, which is a \emph{coassociative} subspace, and to consider $\Lambda^2 (P^{\perp}) = \Lambda^2_+ (P^{\perp}) \oplus \Lambda^2_- (P^{\perp})$, which is isomorphic to $\so(4) = \so(3) \oplus \so(3)$ and is a subspace of $\so(7) = \Lambda^2 (\R^7)$. However, this $\so(4)$ algebra is also \emph{not} the same as $\Theta(P)$. One can check directly that $\Lambda^2_+ (P^{\perp})$ is contained in $\lieg = \Lambda^2_{14}$, whereas neither of the $\so(3)$ subalgebras in $\Theta(P) = \Lambda^2(P) \oplus \Psi(P)$ are contained in $\lieg$.
\end{rmk}


\addcontentsline{toc}{section}{References}

\begin{thebibliography}{99}

\bibitem{A} I. Agricola, ``Old and new on the exceptional group $G_2$'', {\em Notices Amer. Math. Soc.} {\bf 55} (2008), 922--929. MR2441524

\bibitem{BK} J.-P. Bourguignon\ and\ H. Karcher, Curvature operators: pinching estimates and geometric examples, Ann. Sci. \'{E}cole Norm. Sup. (4) {\bf 11} (1978), no.~1, 71--92. MR0493867

\bibitem{Bryant} R. L. Bryant, Some remarks on $G_2$-structures, in {\it Proceedings of G\"{o}kova Geometry-Topology Conference 2005}, 75--109, G\"{o}kova Geometry/Topology Conference (GGT), G\"{o}kova. MR2282011

\bibitem{Harvey} R. Harvey, {\em Spinors and calibrations}, Perspectives in Mathematics, 9, Academic Press, Inc., Boston, MA, 1990. MR1045637

\bibitem{HL} R. Harvey\ and\ H.B. Lawson, ``Calibrated geometries'', {\em Acta Math.} {\bf 148} (1982), 47--157. MR0666108

\bibitem{Iliashenko} A. Iliashenko, ``Betti numbers of nearly $\G$-manifolds with Weyl curvature bounds'', \emph{in preparation}.

\bibitem{K-flows} S. Karigiannis, ``Flows of $\G$-structures. I'', {\em Q. J. Math.} {\bf 60} (2009), 487--522. MR2559631

\bibitem{K-intro} S. Karigiannis, ``Introduction to $\G$ geometry'', in {\it Lectures and surveys on $\G$-manifolds and related topics}, 3--50, {\em Fields Inst. Commun.}, {\bf 84}, Springer, New York. MR4295852

\bibitem{Leung} K.F. Chan\ and\ N.C. Leung, ``Calibrated Submanifolds in $\G$~geometry'', {\em Lectures and Surveys on $\G$-manifolds and related topics}, 103--112, {\em Fields Inst. Commun.}, {\bf 84}, Springer, New York. MR4295852

\bibitem{Lotay} J.D. Lotay, ``Calibrated Submanifolds'', {\em Lectures and Surveys on $\G$-manifolds and related topics}, 69--102, {\em Fields Inst. Commun.}, {\bf 84}, Springer, New York. MR4295852

\bibitem{SW} D.A. Salamon\ and\ T. Walpuski, Notes on the octonions, in {\it Proceedings of the G\"{o}kova Geometry-Topology Conference 2016}, 1--85, G\"{o}kova Geometry/Topology Conference (GGT), G\"{o}kova. MR3676083

\end{thebibliography}
\end{document}